\documentclass[12pt]{amsart}
\usepackage{amsmath,amssymb,latexsym,cancel,rotating}
\usepackage{graphicx,amssymb,mathrsfs,amsmath,color,fancyhdr,amsthm}
\usepackage[all]{xy}
\usepackage{verbatim} 
\usepackage[all]{xy}
\usepackage{graphicx}
\usepackage{rotating} 
\usepackage{tikz-cd}

\newtheorem{theorem}{Theorem}[section]
\newtheorem{lemma}[theorem]{Lemma}
\newtheorem{corollary}[theorem]{Corollary}
\newtheorem{definition}[theorem]{Definition}
\newtheorem{proposition}[theorem]{Proposition}

\newtheorem{remark}[theorem]{Remark}

\newtheorem{example}[theorem]{Example}

  \def\leq{\leqslant}  \def\geq{\geqslant}

\begin{document}

\title[Notes on Chevalley Groups and Root Category III]
{Notes on Chevalley Groups and Root Category \uppercase\expandafter{\romannumeral3}: the Region of Total Positivity}

\author[Buyan Li]{Buyan Li}
\address{Department of Mathematical Sciences, Tsinghua University, Beijing 100084, P. R. China}
\email{liby21@mails.tsinghua.edu.cn}

\author[Jie Xiao]{Jie Xiao}
\address{School of Mathematical Sciences, Beijing Normal University, Beijing 100875, P. R. China}
\email{jxiao@bnu.edu.cn}




\bibliographystyle{abbrv}

\maketitle

\begin{abstract}
    In \cite{notes1}, we use the root categories to realize Chevalley groups.
    Lusztig's theory of total positivity for reductive groups can be naturally applied to Chevalley groups.
    In this paper, we explicitly determine regions of $\mathbb{R}_{>0}^t$ for describing the size of monoids of totally positive elements, with respect to the root subgroups corresponding to the indecomposable objects in the root category.
\end{abstract}

\setcounter{tocdepth}{2}\tableofcontents

\section{Introduction}

Chevalley groups are a family of simple groups over arbitrary fields, constructed from simple complex Lie algebras.
Ringel \cite{RINGEL1990137} realizes the positive part of a Lie algebra using the representation theory of finite dimensional associative hereditary representation-finite algebras, and then Peng and Xiao \cite{1997ROOT} realize the whole Lie algebra using the root category.
Based on their work, we construct the Chevalley groups from the root category in \cite{notes1}.

In \cite{TR1}, Lusztig generalizes the classical definition of totally positivity for $\operatorname{GL}_n(\mathbb{R})$ to any split reductive algebraic group over $\mathbb{R}$.
Although the definitions of monoids of totally positive elements are elementary, they yields nontrivial conclusions and connection with the canonical basis.
This theory can be naturally applied to Chevalley groups.
Since Lusztig's definitions and properties all make use of the root subgroups corresponding to simple roots, we would like to ask what do the elements look like regarding root subgroups corresponding to any root, or equivalently, any indecomposable objects in the root category.

In this paper, we firstly define some regions in $\mathbb{R}_{>0}^m$ which can be viewed as arising from the root category.
Using these regions and an ordering of indecomposable objects of a fixed complete hereditary subcategory of the root category, we characterize the size of the monoid of totally positive elements with respect to the root subgroups corresponding to these indecomposable objects.

\section{Preliminaries}

\subsection{Root Categories and Chevalley Groups}

In this subsection, we recall the construction and some  properties of Chevalley groups arising from the root category, and refer \cite{notes1} for details.

Let $A$ be a finite dimensional associative representation-finite hereditary algebra over some base field $k$, and $\operatorname{mod}A$ be the category of finite dimensional $A$-modules. Let $D^b(A)$ be the bounded derived category of $\operatorname{mod}A$, which is a triangulated category with translantion functor $T$. The root category $\mathcal{R}$ of $A$ is the orbit category $D^b(A)/T^2$, which is also a triangulated category (see \cite{1997ROOT} or \cite{Keller}). 
Let $\operatorname{ind}\mathcal{R}$ be the set of all the indecomposable objects of $\mathcal{R}$. For any object $M\in \mathcal{R}$, we denote its isomorphism class by $[M]$.

Let $\mathcal{K}$ be the Grothendieck group of $\mathcal{R}$, i.e. an abelian group defined by generators $H_{[M]}$, for all $M\in \mathcal{R}$, and subject to relations $H_{[X]}+H_{[Z]}=H_{[Y]}$, if there exists a triangle $X\rightarrow Y\rightarrow Z \rightarrow TX$ in $\mathcal{R}$.
Let $\mathcal{N}$ be a free abelian group with a basis $\{u_{[X]}\}_{X\in \operatorname{ind}\mathcal{R}}$. For simplicity, we write $H_M$ instead of $H_{[M]}$, $u_X$ instead of $u_{[X]}$.

Denote the symmetric Euler form of $\mathcal{R}$ by $(-|-):\mathcal{K}\times \mathcal{K} \rightarrow \mathbb{Z}$. 
Define $d(X)=\operatorname{dim}_k \operatorname{End}_{\mathcal{R}}X$ for any $X\in \operatorname{ind}\mathcal{R}$.
Note that $A_{XY}=\frac{(H_X|H_Y)}{d(X)}$ is an integer.

For any $X,Y,L\in \operatorname{ind}\mathcal{R}$, let $\gamma_{XY}^L=\varphi_{XY}^L(1)-\varphi_{YX}^L(1)$ be the evaluation of the Hall polynomials in $\mathcal{R}$.
All the possible values of $\gamma_{XY}^L$ can be calculated by Ringel's result in \cite{RINGEL1990137}.
Using these to define structure constants, Peng and Xiao \cite{1997ROOT} realized the whole simple Lie algebra $\mathfrak{g}_{\mathbb{C}}=(\mathcal{K}\oplus \mathcal{N})\otimes \mathbb{C}$ corresponding to the root system of $\mathcal{R}$, as well as its $\mathbb{Z}$-form $\mathfrak{g}_{\mathbb{Z}}$.
Since the structure constants of $\mathfrak{g}_{\mathbb{Z}}$ are integers, one can construct Chevalley group as follows (see the definition for Chevalley group also in Carter \cite{carter}, Steinberg \cite{Steinberg}, Geck \cite{geck}).

For any $X\in \operatorname{ind}\mathcal{R}$ and $t\in \mathbb{C}$, define $E_{[X]}(t)=exp(t \operatorname{ad}u_X)$.
Since the coefficients of the action of $E_{[X]}(t)$ on a basis of $\mathfrak{g}_{\mathbb{Z}}$ are all integers, we can extend the definition of $E_{[X]}(t)$ to any field $\mathbb{K}$.
Let $\mathfrak{g}_{\mathbb{K}}=\mathfrak{g}_{\mathbb{Z}}\otimes \mathbb{K}$, and $\mathbf{G}=\mathbf{G}(\mathcal{R},\mathbb{K})$ be the subgroup of $\operatorname{Aut}(\mathfrak{g}_{\mathbb{K}})$ generated by $E_{[X]}(t)$, for all $X\in \operatorname{ind}\mathcal{R}$ and $t\in \mathbb{K}$.
This group is called the Chevalley group of the root category $\mathcal{R}$ over field $\mathbb{K}$.

For $X,Y\in \operatorname{ind}\mathcal{R}$ such that $X\ncong Y$ and $X \ncong TY$, integers $i,j\geq 0$, we inductively define isomorphism classes $[L_{X,Y,i,j}]$ (and their representatives $L_{X,Y,i,j}$) as follows.
Let $[L_{X,Y,1,0}]=[X]$ and $[L_{X,Y,0,1}]=[Y]$. 
If $[L_{X,Y,i,j}]$ is defined, and $L_{X,Y,i,j}$ and $X$ have indecomposable extension, then let $[L_{X,Y,i+1,j}]$ be the isomorphism class of this extension. 
Otherwise, we say $[L_{X,Y,i+1,j}]$ doesn't exist.
Similarly, if  $L_{X,Y,i,j}$ and $Y$ have indecomposable extension, then let $[L_{X,Y,i,j+1}]$ be the isomorphism class of this extension.
Otherwise, we say $[L_{X,Y,i,j+1}]$ doesn't exist.
Let $p_{XY}$ be the largest integer $r$ such that $[L_{TX,Y,r,1}]$ exists, and let $q_{XY}$ be the largest integer $s$ such that $[L_{X,Y,s,1}]$ exists.

For $X,Y\in \operatorname{ind}\mathcal{R}$ such that $X\ncong Y$ and $X \ncong TY$, $ t,s\in \mathbb{K}$, denote the commutator of $E_{[X]}(t)$ and $E_{[Y]}(s)$ by
\begin{align*}
    (E_{[X]}(t),E_{[Y]}(s)) = E_{[X]}(t) E_{[Y]}(s) E_{[X]}(t)^{-1} E_{[Y]}(s)^{-1} .
\end{align*}

\begin{proposition}\cite[Proposition 4.2]{notes1}\label{commutator}
    For $X,Y\in \operatorname{ind}\mathcal{R}$ such that $X\ncong Y$ and $X \ncong TY$, $ t,s\in \mathbb{K}$, we have
    \begin{equation*}
        (E_{[X]}(t),E_{[Y]}(s))= \prod\limits_{i,j>0} E_{[L_{X,Y,i,j}]}(C_{X,Y,i,j}t^i s^j),
    \end{equation*}
where the product is taken over all pairs of $i,j>0$ such that $[L_{X,Y,i,j}]$ exists, and in order of increasing $i+j$.

The constants $C_{X,Y,i,j}$ are given as follows (denote $L_{X,Y,i,j}$ simply by $L_{i,j}$)
\begin{align*}
    &C_{X,Y,1,1} = \gamma_{XY}^{L_{1,1}}, \qquad C_{X,Y,2,1} = \frac{1}{2!}\gamma_{XY}^{L_{1,1}}\gamma_{XL_{1,1}}^{L_{2,1}}, \qquad C_{X,Y,1,2} = -\frac{1}{2!}\gamma_{YX}^{L_{1,1}}\gamma_{YL_{1,1}}^{L_{1,2}}, \\
    &C_{X,Y,3,1} = \frac{1}{3!}\gamma_{XY}^{L_{1,1}}\gamma_{XL_{1,1}}^{L_{2,1}}\gamma_{XL_{2,1}}^{L_{3,1}},\qquad C_{X,Y,3,2} = \frac{1}{3} \gamma_{XY}^{L_{1,1}} \gamma_{XL_{1,1}}^{L_{2,1}} \gamma_{XL_{2,1}}^{L_{3,1}} \gamma_{YL_{3,1}}^{L_{3,2}},\\
    &C_{X,Y,1,3} = - \frac{1}{3!}\gamma_{YX}^{L_{1,1}}\gamma_{YL_{1,1}}^{L_{1,2}}\gamma_{YL_{1,2}}^{L_{1,3}} ,\qquad C_{X,Y,2,3} = \frac{1}{6} \gamma_{YX}^{L_{1,1}} \gamma_{YL_{1,1}}^{L_{1,2}} \gamma_{YL_{1,2}}^{L_{1,3}} \gamma_{XL_{1,3}}^{L_{2,3}}  . 
\end{align*}

\end{proposition}

Define automorphism $h_{[X]}(t)$ of $\mathfrak{g}_{\mathbb{K}}$ for each $X\in \operatorname{ind}\mathcal{R}$ and $t\in \mathbb{K}^{\times}$ by the action
\begin{align*}
    &h_{[X]}(t)u_{Y} = t^{A_{XY}} u_{Y}, \\
    &h_{[X]}(t)H_{S} = H_{S}.
\end{align*}

For $X\in \operatorname{ind}\mathcal{R}$, we use BGP reflection functor to define automorphism $n_{[X]}$ of $\mathfrak{g}_{\mathbb{K}}$. 
Precisely, for any indecomposable object $X$ in $\mathcal{R}$, let $\mathcal{A}$ be a complete hereditary subcategory of $\mathcal{R}$ such that $X$ is projective and simple in $\mathcal{A}$.
Suppose $X$ corresponds to sink $i$ in the corresponding quiver, and denote the BGP reflection functor by $R(S_i^+)$.
Let $\mathcal{A}'=R(S_i^+)\mathcal{A}$.
There exists a unique Lie algebra isomorphism $\phi_{\mathcal{A}}:\mathfrak{g}_{\mathbb{K}}\rightarrow \mathfrak{g}(\Phi(\mathcal{R}))_{\mathbb{K}}$, where $\mathfrak{g}(\Phi(\mathcal{R}))_{\mathbb{K}}=\mathbb{K}\otimes \mathfrak{g}(\Phi(\mathcal{R}))_{\mathbb{Z}}$ and $\mathfrak{g}(\Phi(\mathcal{R}))_{\mathbb{Z}}$ is the $\mathbb{Z}$-span of a Chevalley basis of $\mathfrak{g}(\Phi(\mathcal{R}))$.
Define
\begin{align*}
    n_{[X]}=\phi_{\mathcal{A}'}^{-1}\phi_{\mathcal{A}}:\mathfrak{g}_{\mathbb{K}}\rightarrow \mathfrak{g}_{\mathbb{K}}.
\end{align*}
We can prove that $n_{[X]}=E_{[X]}(1)E_{[TX]}(1)E_{[X]}(1)$.
For any $t\in \mathbb{K}^{\times}$ and $X\in \operatorname{ind}\mathcal{R}$, define $n_{[X]}(t)=h_{[X]}(t)n_{[X]}$. Then we have $n_{[X]}(t)=E_{[X]}(t) E_{[TX]}(t^{-1}) E_{[X]}(t)$, and thus $n_{[X]}(t),h_{[X]}(t)\in \mathbf{G}$.

For each $M,N\in \operatorname{ind}\mathcal{R}$, we introduce a symbol $\omega_M(N)$. 
If $M\cong N$ or $M\cong TN$, let $\omega_M(N)=TN$.
Otherwise, let 
\begin{equation*}
    \omega_M(N) = 
    \begin{cases}
        L_{TM,N,p_{MN}-q_{MN},1}  &  \text{if } p_{MN}-q_{MN} >0 \\
        L_{M,N,q_{MN}-p_{MN},1}  &  \text{if } p_{MN}-q_{MN} \leqslant 0
    \end{cases}
    .
\end{equation*}

Using this symbol, we have
\begin{align*}
    n_{[X]}u_Y=\eta_{XY}u_{\omega_X(Y)},
\end{align*}
for all $X,Y\in \operatorname{ind}\mathcal{R}$, where $\eta_{XY}=\pm 1$ is a product of some $\gamma$'s and some rational number, which can be calculated case by case.
Let $\eta_{XX}=\eta_{X,TX}=1$.
Note that for any $X,Y\in \operatorname{ind}\mathcal{R}$, we have $\eta_{X,TY}=\eta_{XY}$ and $\eta_{XY}\eta_{X,\omega_X(Y)}=(-1)^{A_{XY}}$.

Let $\mathbf{H}$ be the subgroup of $\mathbf{G}$ generated by $h_{[M]}(t)$, for all $M \in \operatorname{ind}\mathcal{R}, t \in \mathbb{K}^{\times}$, and let $\mathbf{N}$ be the subgroup of $\mathbf{G}$ generated by $\mathbf{H}$ and $n_{[M]},\forall M \in \operatorname{ind}\mathcal{R}$.
We have $\mathbf{H}$ is a normal subgroup of $\mathbf{N}$, and $\mathbf{N}/\mathbf{H}\cong \mathbf{W}$, where $\mathbf{W}$ is the Weyl group associated to the root system $\Phi$ of $\mathcal{R}$.

Then we summarize six conjugate relations as follows, which are useful in the calculations.

\begin{lemma}
    \cite[Lemma 4.7-4.12]{notes1} \label{conjugate}

    For any $ X,Y \in \operatorname{ind}\mathcal{R},  t \in \mathbb{K}^{\times}, s \in \mathbb{K}$,

    (1) $ n_{[X]}(t) E_{[Y]}(s) n_{[X]}(t)^{-1} = E_{[\omega_X(Y)]}(\eta_{XY}t^{-A_{XY}}s)$;

    (2) $h_{[X]}(t) E_{[Y]}(s) h_{[X]}(t)^{-1} = E_{[Y]}(t^{A_{XY}}s).$

    For any $X,Y \in \operatorname{ind}\mathcal{R},  t,s \in \mathbb{K}^{\times}$,

    (3) $ n_{[X]}(t) n_{[Y]}(s) n_{[X]}(t)^{-1} = n_{[\omega_X(Y)]}(\eta_{XY}t^{-A_{XY}}s) $;

    (4) $ n_{[X]}(t) h_{[Y]}(s) n_{[X]}(t)^{-1} = h_{[\omega_X(Y)]}(s)$;

    (5) $  h_{[X]}(t) h_{[Y]}(s) h_{[X]}(t)^{-1} = h_{[Y]}(s)$;

    (6) $ h_{[X]}(t) n_{[Y]}(s) h_{[X]}(t)^{-1} = n_{[Y]}(t^{A_{XY}}s).$
\end{lemma}

Now we arbitrarily fix a complete section of the root category $\mathcal{R}$, i.e. a connected full subquiver of the AR-quiver of $\mathcal{R}$, which contains one vertex for each $\tau$-orbit. Then we get a complete hereditary subcategory $\mathcal{B}$ of $\mathcal{R}$.
We fix a set of representatives $\{S_1,\cdots,S_n\}$ of isomorphism classes of simple objects in $\mathcal{B}$.

Note that $\mathbf{H}$ is an abelian subgroup of $\mathbf{G}$, and each element in $\mathbf{H}$ is of the form $\prod\limits_{i=1}^n h_{[S_i]}(t_i)$, $t_i\in \mathbb{K}^{\times}$.
Moreover, we have 
 $ \prod\limits_{i=1}^n h_{[S_i]}(t_i) =1$ if and only if 
       $ (t_1,\cdots,t_n) \in (\mathbb{K}^{\times})^n$ satisfies $\prod\limits_{i=1}^n t_i^{a_{ij}} = 1$, for all $ j=1,\cdots,n$,
    where $(a_{ij})$ is the Cartan matrix.

Let $\mathbf{U}^+$ (resp. $\mathbf{U}^-$) be the subgroup of $\mathbf{G}$ generated by $E_{[X]}(t)$, for all $t\in \mathbb{K}$ and $X\in \operatorname{ind}\mathcal{B}$ (resp. $X\in \operatorname{ind}T\mathcal{B}$).
Let $\mathbf{B}^+$ (resp. $\mathbf{B}^-$) be the subgroup of $\mathbf{G}$ generated by $\mathbf{H}$ and $\mathbf{U}^+$ (resp. $\mathbf{U}^-$).

\begin{proposition}
    Every element of $\mathbf{U}^+$ can be uniquely expressed in the form 
    \begin{align*}
        \prod\limits_{M\in \operatorname{ind}\mathcal{B}} E_{[M]}(t_M),
    \end{align*}
    where the product is taken over all indecomposable objects in $\mathcal{B}$ in increasing order (with respect to an order that is compatible with the length).
    Similar for $\mathbf{U}^-$.
\end{proposition}

For any $n\in \mathbf{N}$, denote its image in $\mathbf{N}/\mathbf{H}$ by $\overline{n}$, and define its length
\begin{align*}
    l(\overline{n}) = min\{ t \geq 0 | \overline{n} = \overline{n_{[S_{i_1}]}} \cdots \overline{n_{[S_{i_t}]}}, S_{i_1},\cdots,S_{i_t} \text{ are simple objects in } \mathcal{B} \}.
\end{align*}
A reduced expression of $\overline{n}$ is a sequence $(i_1,\cdots,i_t)$ such that $\overline{n} = \overline{n_{[S_{i_1}]}} \cdots \overline{n_{[S_{i_t}]}}$ and $t=l(\overline{n})$.
For a reduced expression $(i_1,\cdots,i_t)$ of $\overline{n}$, there is a well-defined set
\begin{align*}
    R(\overline{n}) = \{  [M] | M\in \operatorname{ind}\mathcal{B}, \omega_{S_{i_1}}\cdots \omega_{S_{i_t}}(M) \in \operatorname{ind}T\mathcal{B}    \}.
\end{align*}
We have $R(\overline{n}) = \{ [S_{i_t}], [\omega_{S_{i_t}}(S_{i_{t-1}})], \cdots,[\omega_{S_{i_t}} \cdots \omega_{S_{i_2}}(S_{i_1})]   \}$.

For any $n\in \mathbf{N}$, let $ \mathbf{U}^+_{\overline{n}} = \prod\limits_{M\in  R(\overline{n})} \mathbf{E}_{[M]}$, where $\mathbf{E}_{[M]}$ denotes the root subgroup generated by $E_{[M]}(t)$ for all $t\in \mathbb{K}$. (Note that in \cite{notes1} this is denoted by $U_{\overline{n}}^-$, and $U_{\overline{n}}^+$ represents a different product).

\begin{proposition}
    (Bruhat decomposition)

    (1) We have
     \begin{align*}
         \mathbf{G} = \bigsqcup\limits_{\overline{n}\in \mathbf{N}/\mathbf{H}} \mathbf{B}^+ n\mathbf{B}^+ .
    \end{align*}

    (2) Any $x\in \mathbf{G}$ has a unique expression of the form $x=u'hnu$ with $u'\in \mathbf{U}^+$, $h\in \mathbf{H}$, $n\in \mathbf{N}$, $u\in \mathbf{U}^+_{\overline{n}}$.
\end{proposition}

\begin{proposition}
    \cite[Cor 4.42,4.44]{notes1}
    Assume $\mathbb{K}$ is an algebraically closed field, then $\mathbf{G}$ is a semisimple linear algebraic group, and the Lie algebra of $\mathbf{G}$ is isomorphic to $\mathfrak{g}_{\mathbb{K}}$.
\end{proposition}

\subsection{Total Positivity and Reductive Groups}
For any split reductive connected algebraic group $G$ over some infinite field $K$, Lusztig define the submonoids $G_{\geq 0}$ and $G_{>0}$ in \cite{TR1},\cite{TR2}, which generalize the totally positive matrices in $\operatorname{GL}_n(\mathbb{R})$.
He studies these totally positive elements of $G$ using canonical bases and their positivity properties, and obtains various results, see for instance \cite{TC}.
Here we recall the definitions and some properties of these monoids.

Let $K=\mathbb{R}$ and let $G$ be a split reductive connected algebraic group over $\mathbb{R}$. 
We assume that $G$ is simply-laced.
Fix an $\mathbb{R}$-split maximal torus $T$ of $G$, and $W=N_G(T)/T$ is the Weyl group of $G$.
The element of maximal length in $W$ is denoted by $w_0$.
Let $B^+,B^-$ be a pair of opposed Borel subgroups containing $T$. 
Let $U^+$ (resp. $U^-$) be the unipotent radical of $B^+$ (resp. $B^-$), and $U^+_i$ (resp. $U^-_i$), $i\in I$ be the corresponding simple root subgroups of $U^+$ (resp. $U^-$), where $I$ is given by the root datum.
For each $i\in I$, let $\chi_i:\mathbb{R}^{\times}\rightarrow T$ be the simple coroot, and let $x_i:\mathbb{R}\rightarrow U^+_i$ and $y_i:\mathbb{R}\rightarrow U^-_i$ be the algebraic group isomorphisms, such that there is a homomorphism $\operatorname{SL}_2(\mathbb{R})\rightarrow G$ maps 
$\begin{pmatrix}
    1 & a \\ 0 & 1
\end{pmatrix}$
to $x_i(a)$, 
$\begin{pmatrix}
    b & 0 \\ 0 & b^{-1}
\end{pmatrix}$
to $\chi_i(b)$, and
$\begin{pmatrix}
    1 & 0 \\ c & 1
\end{pmatrix}$ 
to $y_i(c)$, for any $a,c\in \mathbb{R}, b\in \mathbb{R}^{\times}$.

Let $U^+_{\geq 0}$ (resp. $U^-_{\geq 0}$) be the submonoid of $U^+$ (resp. $U^-$) generated by $x_i(a)$ (resp. $y_i(a)$), for all $i\in I, a\in \mathbb{R}_{\geq 0}$.
Let $T_{>0}$ be the submonoid of $T$ generated by $\chi_i(a)$, for all $i\in I, a\in \mathbb{R}_{>0}$.
Let $G_{\geq 0}$ be the submonoid of $G$ generated by $U^+_{\geq 0}$, $U^-_{\geq 0}$ and $T_{>0}$.

Lusztig also gives definitions of these monoids by generators and relations.

\begin{definition}\label{u>=0}
    \cite[2.9]{TR2}
    Let $\mathcal{U}_{\geq 0}$ be the monoid with generators the symbols $i^a$, for all $i\in I,a\in \mathbb{R}_{\geq 0}$, and subject to relations:

    (1) For any $i\in I, a,b\in \mathbb{R}_{\geq 0}$, $i^ai^b=i^{a+b}$;

    (2) For $i\neq j$ in $I$, $i^{a_1}j^{a_2}i^{a_3}\cdots = j^{a_1'}i^{a_2'}j^{a_3'}\cdots$, both products have $t=o(s_is_j)$ (the order of $s_is_j$) factors, and $(a_1,\cdots,a_t),(a_1',\cdots,a_t')\in \mathbb{R}_{>0}^t$ satisfies some explicit bijection (See \cite[Proposition 2.5]{TR1} for case of type ADE). 
\end{definition}

Then there are isomorphisms of monoids $U^+_{\geq 0}\cong   \mathcal{U}_{\geq 0}\cong U^-_{\geq 0} $.

\begin{definition}\label{G>=0}
    \cite[2.10]{TR2}
    Let $\mathcal{G}_{\geq 0}$ be the monoid with generators the symbols $i^a,(-i)^a,\underline{i}^b$ for all $i\in I$, $a\in \mathbb{R}_{\geq 0}$ and $b\in \mathbb{R}_{>0}$, and subject to relations:

    (1) For $i\in I, \epsilon =\pm 1, a,b\in \mathbb{R}_{\geq 0}$, we have $(\epsilon i)^a (\epsilon i)^b = (\epsilon i)^{a+b}$.

    (2) For $i\neq j$ in $I$, $(\epsilon i)^{a_1}(\epsilon j)^{a_2}(\epsilon i)^{a_3}\cdots = (\epsilon j)^{a_1'}(\epsilon i)^{a_2'}(\epsilon j)^{a_3'}\cdots$, both products have $t=o(s_is_j)$ factors, and $(a_1,\cdots,a_t),(a_1',\cdots,a_t')\in \mathbb{R}_{>0}^t$ satisfies some explicit bijection (the same as Definition \ref{u>=0}(2)).

    (3) For $i\in I, a,c\in \mathbb{R}_{\geq 0}, b\in \mathbb{R}_{>0}$, we have 
    \begin{align*}
        i^a \underline{i}^b (-i)^c = (-i)^{\frac{c}{ac+b^2}} \underline{i}^{\frac{ac+b^2}{b}} i^{\frac{a}{ac+b^2}}.
    \end{align*}

    (4) For $i,j\in I$ and $a,b\in \mathbb{R}_{>0}$, we have $\underline{i}^a \underline{i}^b = \underline{i}^{ab}$, $\underline{i}^1=1$, $\underline{i}^a \underline{j}^b = \underline{j}^b \underline{i}^a$.

    (5) For $i,j\in I$, $\epsilon =\pm 1$, $a\in \mathbb{R}_{>0}$ and $b\in \mathbb{R}_{\geq 0}$, we have 
    \begin{align*}
        \underline{j}^a (\epsilon i)^b = (\epsilon i)^{a^{\epsilon a_{ji}}b}\underline{j}^a,
    \end{align*}
    where $(a_{ij})$ is the Cartan matrix corresponds to the root datum.

    (6) For $i\neq j$ in $ I$, $\epsilon =\pm 1$, $a,b\in \mathbb{R}_{>0}$, we have $(\epsilon i)^a (-\epsilon j)^b = (-\epsilon j)^b (\epsilon i)^a$.
\end{definition}

Then there is a isomorphism of monoids $\mathcal{G}_{\geq 0}\cong G_{\geq 0}$.

\begin{lemma}\label{utu}
    \cite[Lemma 2.3]{TR1}

    (1) Any element $g\in G_{\geq 0}$ can be written uniquely in the form $g=u^+ t u^-$ with $u^+\in U^+_{\geq 0}$, $t\in T_{>0}$ and $u^-\in U^-_{\geq 0}$.

    (2) Any element $g\in G_{\geq 0}$ can be written uniquely in the form $g=u^- t u^+$ with $u^+\in U^+_{\geq 0}$, $t\in T_{>0}$ and $u^-\in U^-_{\geq 0}$.
\end{lemma}

For each $w\in W$, let $s_{i_1}\cdots s_{i_t}$ be a reduced expression for $w$. Define a subset $U^+(w)$ of $U^+_{\geq 0}$ by 
\begin{align*}
    U^+(w)=\{x_{i_1}(a_1)\cdots x_{i_t}(a_t)| a_1,\cdots,a_t \in \mathbb{R}_{>0}\}.
\end{align*}
It is well-defined and has the following properties.

\begin{proposition}\label{uw}
    \cite[Proposition 2.7, Corollary 2.8]{TR1}
    
    (1) For $w\in W$ with a reduced expression $s_{i_1}\cdots s_{i_t}$, the map $\mathbb{R}^t_{>0}\rightarrow U^+$ given by $(a_1,\cdots, a_t)\mapsto x_{i_1}(a_1)\cdots x_{i_t}(a_t)$ is injective, and its image doesn't depend on the choice of reduced expression for $w$.

    (2) The set $U^+(w)$ is a subset of $B^- \dot{w} B^-$, where $\dot{w}$ is a representative of $w$ in $N_G(T)$.

    (3) For $w,w'\in W$ and $w\neq w'$, we have $U^+(w) \cap U^+(w')=\emptyset$.

    (4) We have $U^+_{\geq 0}=\bigsqcup\limits_{w\in W}U^+(w)$.
\end{proposition}

Similar results hold for $U^-$.

Define $U^+_{>0}=U^+(w_0)$, $U^-_{>0}=U^-(w_0)$, and $G_{>0}=U^+_{>0} T_{>0} U^-_{>0}=U^-_{>0} T_{>0} U^+_{>0}$.
Then $G_{>0}$ is a submonoid of $G_{\geq 0}$.

\begin{theorem}\label{dense}
    \cite[Proposition 4.2, Theorem 4.3, Remark 4.4]{TR1}

    (1) The subset $U^+_{>0}$ is dense in $U^+_{\geq 0}$, and the subset $U^+_{\geq 0}$ is closed in $U^+$. Similar for $U^-$.

    (2) The subset $G_{>0}$ is dense in $G_{\geq 0}$, and the subset $G_{\geq 0}$ is closed in $G$. 
\end{theorem}

Let $\mathfrak{B}$ be the set of all Borel subgroups of $G$.
For $x\in G$ and a Borel subgroup $B$, denote $xBx^{-1}$ by $x\cdot B$.
Let $s_{i_1}\cdots s_{i_m}$ be a reduced expression for $w_0$ and denote by $\mathbf{i}=(i_1,\cdots,i_m)$.
For $0\leq k\leq m$, Lusztig defines a map $f_{\mathbf{i},k}:\mathbb{R}_{>0}^m\rightarrow \mathfrak{B}$ by 
    \begin{align*}
        f_{\mathbf{i},k}(a_1,\cdots,a_m)=x_{i_1}(a_1)\cdots x_{i_k}(a_k) y_{i_m}(a_m)\cdots y_{i_{k+1}}(a_{k+1}) \dot{s}_{i_{k+1}}\cdots \dot{s}_{i_m} \cdot B^-.
    \end{align*}
and shows that these maps are injective and independent of the choice of $\mathbf{i}$ and $k$ in \cite[3.1, Proposition 3.2]{TC}.
In particular, considering the image of $f_{\mathbf{i},0}$ and $f_{\mathbf{i},m}$, we have the following result.

\begin{proposition}\label{u+-bij}
    \cite[Proposition 3.3]{TC}

    There exists a unique bijection $\phi: U^+_{>0}\rightarrow U^-_{>0}$ such that $\phi(u)B^+\phi(u)^{-1}=uB^- u^{-1}$ for all $u\in U^+_{>0}$.
\end{proposition}

\section{Total Positivity and $\mathbf{G}(\mathcal{R})$}

In this paper, we aim to characterize the size of these monoids of totally positive elements, using the root subgroups corresponding to indecomposable objects in the root category $\mathcal{R}$. 
We keep the notations from the previous section and assume that $\mathcal{R}$ is of type ADE.
The results actually hold for any finite type by means of folding.

\subsection{Submonoids in $\mathbf{G}(\mathcal{R})$}

We consider Lusztig's constructions in the Chevalley group $\mathbf{G}=\mathbf{G}(\mathcal{R},\mathbb{R})$ from the root category.

\begin{lemma}
    For any $X\in \operatorname{ind}\mathcal{R}$, $a,c\geq 0$ and $b>0$, we have
    \begin{align*}
        E_{[X]}(a) h_{[X]}(b) E_{[TX]}(-c) = E_{[TX]}(-\frac{c}{ac+b^2}) h_{[X]}(\frac{ac+b^2}{b}) E_{[X]}(\frac{a}{ac+b^2}).
    \end{align*}
\end{lemma}

\begin{proof}
    This Lemma follows from 
    \begin{align*}
        \begin{pmatrix}
            1 & a \\ 0 & 1 
        \end{pmatrix}
        \begin{pmatrix}
            b & 0 \\ 0 & b^{-1}
        \end{pmatrix}
        \begin{pmatrix}
            1 & 0 \\ c & 1
        \end{pmatrix} 
        = 
        \begin{pmatrix}
            1 & 0 \\ \frac{c}{ac+b^2} & 1
        \end{pmatrix}
        \begin{pmatrix}
            \frac{ac+b^2}{b} & 0 \\ 0 & \frac{b}{ac+b^2}
        \end{pmatrix}
        \begin{pmatrix}
            1 & \frac{a}{ac+b^2} \\ 0 & 1
        \end{pmatrix}
    \end{align*}
    and the group homomorphism $\operatorname{SL}_2(\mathbb{R})\rightarrow \mathbf{G}$ mapping $\begin{pmatrix}
            1 & a \\ 0 & 1 
        \end{pmatrix}$ to $E_{[X]}(a)$, $\begin{pmatrix}
            1 & 0 \\ -a & 1
        \end{pmatrix} $ to $E_{[TX]}(a)$ and $\begin{pmatrix}
            a & 0 \\ 0 & a^{-1}
        \end{pmatrix}$ to $h_{[X]}(a)$.
\end{proof}

\begin{lemma}
    For $X,Y\in \operatorname{ind}\mathcal{R}$, $X\ncong Y,X\ncong TY$, and $a,b,c\in \mathbb{R}$, 

    (1) If $\overline{n}_{[X]}\overline{n}_{[Y]}=\overline{n}_{[Y]}\overline{n}_{[X]}$, then $E_{[X]}(a)E_{[Y]}(b)=E_{[Y]}(b)E_{[X]}(a)$;

    (2) If $\overline{n}_{[X]}\overline{n}_{[Y]}\overline{n}_{[X]}=\overline{n}_{[Y]}\overline{n}_{[X]}\overline{n}_{[Y]}$ and $a+c\neq 0$, then 
    \begin{align*}
        E_{[X]}(a) E_{[Y]}(b) E_{[X]}(c) = E_{[Y]}(\frac{bc}{a+c}) E_{[X]}(a+c) E_{[Y]}(\frac{ab}{a+c}).
    \end{align*}
\end{lemma}

\begin{proof}
    (1) If $\overline{n}_{[X]}\overline{n}_{[Y]}=\overline{n}_{[Y]}\overline{n}_{[X]}$, then $\overline{n}_{[\omega_Y(X)]}=\overline{n}_{[X]}$, which implies $[\omega_Y(X)]=[X]$ (By the definition of $\omega_Y(X)$, we exclude the case that $[\omega_Y(X)]=[TX]$).
    This means there doesn't exist indecomposable extension of $X$ and $Y$, and by the commutator formula, we have $E_{[X]}(a)E_{[Y]}(b)=E_{[Y]}(b)E_{[X]}(a)$.

    (2) If $\overline{n}_{[X]}\overline{n}_{[Y]}\overline{n}_{[X]}=\overline{n}_{[Y]}\overline{n}_{[X]}\overline{n}_{[Y]}$, then $\overline{n}_{[\omega_Y(X)]}=\overline{n}_{[\omega_X(Y)]}$. 
    Similarly, we have $[\omega_Y(X)]=[\omega_X(Y)]$.
    By the assumption that $\mathcal{R}$ is of type ADE, this implies that there exists an indecomposable extension $L=L_{X,Y,1,1}$ of $X$ and $Y$, and the commutator relation is $(E_{[X]}(a),E_{[Y]}(b))=E_{[L]}(\gamma_{XY}^Lab)$. Then 
    \begin{align*}
        E_{[X]}(a) E_{[Y]}(b) E_{[X]}(c) &= E_{[Y]}(b) E_{[X]}(a+c) E_{[L]}(\gamma_{XY}^L ab)\\
         &= E_{[Y]}(\frac{bc}{a+c}) E_{[X]}(a+c) E_{[Y]}(\frac{ab}{a+c}).
    \end{align*}
    The Lemma is proved.
\end{proof}

We arbitrarily fix a complete section of $\mathcal{R}$, and get a complete hereditary subcategory $\mathcal{B}$ of $\mathcal{R}$.
Let $\{S_1,\cdots,S_n\}$ be a set of representatives of isomorphism classes of simple objects in $\mathcal{B}$.

Let $\mathbf{U}^+_{\geq 0}$ be the submonoid of $\mathbf{U}^+$ generated by $E_{[S_i]}(t)$ for $t\geq 0$, $i=1,\cdots,n$.
Let $\mathbf{U}^-_{\geq 0}$ be the submonoid of $\mathbf{U}^-$ generated by $E_{[TS_i]}(-t)$ for $t\geq 0$, $i=1,\cdots,n$.
Let $\mathbf{H}_{>0}$ be the submonoid of $\mathbf{H}$ generated by $h_{[S_i]}(t)$ for $t>0$, $i=1,\cdots,n$.
Let $\mathbf{G}_{\geq 0}$ be the submonoid of $\mathbf{G}$ generated by $E_{[S_i]}(t), E_{[TS_i]}(-t)$ for $t\geq 0$, and $h_{[S_i]}(t)$ for $t>0$, $i=1,\cdots,n$.
Then $\mathbf{G}_{\geq 0}=\mathbf{U}^+_{\geq 0}\mathbf{H}_{>0}\mathbf{U}^-_{\geq 0}=\mathbf{U}^-_{\geq 0}\mathbf{H}_{>0}\mathbf{U}^+_{\geq 0}$, and Lemma \ref{utu} holds for $\mathbf{G}_{\geq 0},\mathbf{U}^+_{\geq 0},\mathbf{H}_{>0},\mathbf{U}^-_{\geq 0}$ replacing $G_{\geq 0},U^+_{\geq 0},T_{>0},U^-_{\geq 0}$.

For any $n\in \mathbf{N}$, let $(i_1,\cdots,i_t)$ be a reduced expression for $\overline{n}$.
Define 
\begin{align*}
    &\mathbf{U}^+(\overline{n})=\{ E_{[S_{i_1}]}(a_1)\cdots E_{[S_{i_t}]}(a_t)\mid a_1,\cdots,a_t \in \mathbb{R}_{>0}  \},\\
     &\mathbf{U}^-(\overline{n})=\{ E_{[TS_{i_1}]}(-a_1)\cdots E_{[TS_{i_t}]}(-a_t)\mid a_1,\cdots,a_t \in \mathbb{R}_{>0}  \}.
\end{align*}
Then we have Proposition \ref{uw} holds for $\mathbf{U}^+(\overline{n})$ replacing $U^+(w)$.
Similar for $\mathbf{U}^-(\overline{n})$.
We can easily deduce that $\mathbf{U}^+_{\geq 0}\cong \mathcal{U}_{\geq 0}\cong \mathbf{U}^-_{\geq 0}$.

\begin{proposition}
    There is an isomorphism of monoids $\psi:\mathcal{G}_{\geq 0}\rightarrow \mathbf{G}_{\geq 0}$.
\end{proposition}

\begin{proof}
    Define $\psi(i^a)=E_{[S_i]}(a),\psi((-i)^a)=E_{[TS_i]}(-a),\psi(\underline{i}^a)=h_{[S_i]}(a)$. 
    Since the corresponding generators of $\mathbf{G}_{\geq 0}$ satisfies the relations in Definition \ref{G>=0}, $\psi$ is a well-defined homomorphism of monoids, and is surjective. 

    The surjective map $\mathbb{R}^n \rightarrow \mathbf{H}$ restricts to a surjective map $\mathbb{R}^n_{>0}\rightarrow \mathbf{H}_{>0}$, mapping $(a_1,\cdots,a_n)$ to $h_{[S_1]}(a_1)\cdots h_{[S_n]}(a_n)$.
    If there exists $(t_1,\cdots,t_n)\in \mathbb{R}^n_{>0}$ such that $h_{[S_1]}(t_1)\cdots h_{[S_n]}(t_n)=1$, then $\prod\limits_{i=1}^n t_i^{a_{ij}} = 1$, for all $ j=1,\cdots,n$.
    Since $t_1,\cdots,t_n>0$, we can take the logarithm of both sides of the equality, and obtain that 
    \begin{align*}
        0=\operatorname{log}(\prod\limits_{i=1}^n t_i^{a_{ij}})=\sum_{i=1}^n a_{ij}\operatorname{log} t_i.
    \end{align*}  
    Since the Cartan matrix $(a_{ij})$ is invertible, so does its transpose.
    Then the system of equations has only the trivial solution, that is, $\operatorname{log}t_i=0$ for all $i=1,\cdots,n$. 
    Hence $t_i=1$ for all $i=1,\cdots,n$, which means the restriction $\mathbb{R}^n_{>0}\rightarrow \mathbf{H}_{>0}$ is bijective.

    Then we have a bijection $\bigsqcup_{\overline{n},\overline{n}'\in \mathbf{N}/\mathbf{H}} \mathbb{R}_{>0}^{l(\overline{n})+l(\overline{n}')+n}\rightarrow \mathbf{G}_{\geq 0}$, which coincides with the composition of the bijection $\bigsqcup_{w,w'\in W} \mathbb{R}_{>0}^{l(w)+l(w')+n}\rightarrow G_{\geq 0}$, the isomorphism $G_{\geq 0}\rightarrow \mathcal{G}_{\geq 0}$ and $\psi$ (identifying $\mathbf{N}/\mathbf{H}$ and $W$).
    This shows that $\psi$ is injective, and thus an isomorphism of monoids. 
    The Proposition is proved.
\end{proof}

We fix an $n_0\in \mathbf{N}$ such that its length $l(\overline{n}_0)$ is the largest.
Let $\mathbf{U}^+_{>0}=\mathbf{U}^+(\overline{n}_0)$, $\mathbf{U}^-_{>0}=\mathbf{U}^-(\overline{n}_0)$, and $\mathbf{G}_{>0}=\mathbf{U}^+_{>0} \mathbf{H}_{>0}\mathbf{U}^-_{>0}$.
The Theorem \ref{dense} and the Proposition \ref{u+-bij} hold for $\mathbf{U}^+_{>0},\mathbf{U}^+_{\geq 0},\mathbf{U}^+,\mathbf{U}^-_{>0},\mathbf{U}^-_{\geq 0},\mathbf{U}^-,\mathbf{G}_{>0},\mathbf{G}_{\geq 0},\mathbf{G}$, replacing $U^+_{>0},U^+_{\geq 0},U^+$, $U^-_{>0},U^-_{\geq 0},U^-,G_{>0},G_{\geq 0},G$.

\subsection{Fundamental Region from Root Category}

In this subsection, we define some regions in $\mathbb{R}_{>0}^m$, which can be regarded as obtained from the root category (via bijections).

Again, we arbitrarily fix a complete section of $\mathcal{R}$, and get a complete hereditary subcategory $\mathcal{B}$ of $\mathcal{R}$.
Let $\{S_1,\cdots,S_n\}$ be a set of representatives of isomorphism classes of simple objects in $\mathcal{B}$.
Let $(i_1,\cdots,i_m)$ be a reduced expression of $\overline{n}_0$. 
For $2\leq k \leq m$, define 
\begin{align*}
    \varepsilon_k = \eta_{S_{i_1},\omega_{S_{i_1}}\cdots \omega_{S_{i_{k-1}}}(S_{i_k})} \eta_{S_{i_{2}},\omega_{S_{i_{2}}}\cdots \omega_{S_{i_{k-1}}}(S_{i_k})} \cdots \eta_{S_{i_{k-1}},\omega_{S_{i_{k-1}}}(S_{i_k})},
\end{align*}
and let $\varepsilon_1=1$.
Or equivalently, let $\varepsilon_k\in \{\pm 1\}$ such that 
\begin{align*}
    n_{[S_{i_{k-1}}]}\cdots n_{[S_{i_1}]} u_{\omega_{S_{i_1}}\cdots \omega_{S_{i_{k-1}}}(S_{i_k})} = \varepsilon_k u_{S_{i_k}}.
\end{align*}
We call these $\varepsilon_k$ the Tits signs for $(i_1,\cdots,i_m)$.

For simplicity, we write $n_i$ for $n_{[S_i]}$, for $i=1,\cdots,n$.

\begin{lemma}
    For $i\neq j$ in $\{1,\cdots,n\}$, we have 

    (1) If $\overline{n}_i\overline{n}_j=\overline{n}_j\overline{n}_i$, then $n_in_j=n_jn_i$;

    (2) If $\overline{n}_i\overline{n}_j\overline{n}_i=\overline{n}_j\overline{n}_i\overline{n}_j$, then $n_in_jn_i=n_jn_in_j$.
\end{lemma}

\begin{proof}
    (1) If $\overline{n}_i\overline{n}_j=\overline{n}_j\overline{n}_i$, then there doesn't exist indecomposable extension of $S_i$ and $S_j$.
    By Proposition \ref{commutator}, we have 
    \begin{align*}
        n_in_j&=E_{[S_i]}(1)E_{[TS_i]}(1)E_{[S_i]}(1)E_{[S_j]}(1)E_{[TS_j]}(1)E_{[S_j]}(1)\\
        &=E_{[S_j]}(1)E_{[TS_j]}(1)E_{[S_j]}(1)E_{[S_i]}(1)E_{[TS_i]}(1)E_{[S_i]}(1)=n_jn_i.
    \end{align*}

    (2) If $\overline{n}_i\overline{n}_j\overline{n}_i=\overline{n}_j\overline{n}_i\overline{n}_j$, then by Lemma \ref{conjugate} we have 
    \begin{align*}
        n_j^{-1} n_i^{-1} (n_j^{-1} n_i n_j) n_i &= n_j^{-1} (n_i^{-1} n_{[\omega_{S_j}(S_i)]}(-\eta_{S_j,S_i}) n_i)\\
        &=  n_j^{-1} n_{[\omega_{S_i}\omega_{S_j}(S_i)]}(\eta_{S_i,\omega_{S_j}(S_i)}\eta_{S_j,S_i})\\
        &= n_j^{-1} n_j(\eta_{S_i,\omega_{S_j}(S_i)}\eta_{S_j,S_i}) = n_j^{-1}n_j =1,
    \end{align*}
    where $\eta_{S_i,\omega_{S_j}(S_i)}=\gamma_{TS_i,L_{S_i,S_j,1,1}}^{S_j}=\gamma_{S_j,S_i}^{L_{S_i,S_j,1,1}}$, $\eta_{S_j,S_i}=\gamma_{S_j,S_i}^{L_{S_i,S_j,1,1}}$, and $(\gamma_{S_j,S_i}^{L_{S_i,S_j,1,1}})^2=1$.
    Thus $n_in_jn_i=n_jn_in_j$. The Lemma is proved.
\end{proof}

\begin{lemma}
    For two reduced expressions $(i_1,\cdots,i_m)$ and $(i_1',\cdots,i_m')$ of $\overline{n}_0$, denote the corresponding Tits signs by $\varepsilon_k$ and $\varepsilon_k'$ respectively.

    (1) If $i_k'=i_{k+1},i_{k+1}'=i_k$ for some $k$, and $i_l'=i_l$ for all $l\neq k,k+1$, then we have $\varepsilon_k'=\varepsilon_{k+1},\varepsilon_{k+1}'=\varepsilon_k$, and $\varepsilon_l'=\varepsilon_l$ for all $l\neq k,k+1$.

    (2) If $(i_k,i_{k+1},i_{k+2})=(i,j,i)$, $(i_k',i_{k+1}',i_{k+2}')=(j,i,j)$ for some $k$ and $i,j$, and $i_l'=i_l$ for all $l\neq k,k+1,k+2$, then we have
    \begin{align*}
        \varepsilon_k'=\varepsilon_{k+2}, \quad \varepsilon_{k+1}'=-\varepsilon_{k+1}, \quad \varepsilon_{k+2}'=\varepsilon_k,
    \end{align*}
    and $\varepsilon_l'=\varepsilon_l$ for all $l\neq k,k+1,k+2$.
\end{lemma}

\begin{proof}
    (1) In this case, $\overline{n}_{i_k}\overline{n}_{i_{k+1}}=\overline{n}_{i_{k+1}}\overline{n}_{i_k}$. Then $\eta_{S_{i_k},S_{i_{k+1}}}=\eta_{S_{i_{k+1}},S_{i_{k}}}=1$, and by definition we have $\varepsilon_k'=\varepsilon_{k+1},\varepsilon_{k+1}'=\varepsilon_k$.
    It's obvious that $\varepsilon_l'=\varepsilon_l$ for $l<k$. As for $l>k+1$, by the previous Lemma and definition, we have
    \begin{align*}
        \varepsilon_l u_{S_{i_l}} &= n_{i_{l-1}}\cdots n_{i_{k+1}} n_{i_k} \cdots n_{i_1} u_{\omega_{S_{i_1}}\cdots \omega_{S_{i_k}}\omega_{S_{i_{k+1}}} \cdots \omega_{S_{i_{l-1}}}(S_{i_l})} \\
        &= n_{i_{l-1}}\cdots n_{i_{k}} n_{i_{k+1}} \cdots n_{i_1} u_{\omega_{S_{i_1}}\cdots \omega_{S_{i_{k+1}}}\omega_{S_{i_{k}}} \cdots \omega_{S_{i_{l-1}}}(S_{i_l})} = \varepsilon_l' u_{S_{i_l}},
    \end{align*}
    and thus $\varepsilon_l'=\varepsilon_l$.

    (2) In this case, we can calculate that 
    \begin{align*}
        &\eta_{S_i,S_j}=\gamma_{S_i,S_j}^{L_{S_i,S_j,1,1}}, \quad \eta_{S_j,S_i}=\gamma_{S_j,S_i}^{L_{S_i,S_j,1,1}}=-\gamma_{S_i,S_j}^{L_{S_i,S_j,1,1}},\\
        &\eta_{S_j,L_{S_i,S_j,1,1}}=\gamma_{TS_j,L_{S_i,S_j,1,1}}^{S_i}=\gamma_{S_i,S_j}^{L_{S_i,S_j,1,1}}, \quad \eta_{S_i,L_{S_i,S_j,1,1}}=\gamma_{S_j,S_i}^{L_{S_i,S_j,1,1}}.
    \end{align*}
    Thus 
    \begin{align*}
       & \varepsilon_k'=\varepsilon_{k+2}(\gamma_{S_j,S_i}^{L_{S_i,S_j,1,1}})^2=\varepsilon_{k+2},\\
        &  \varepsilon_{k+1}'=\varepsilon_{k+1}\gamma_{S_j,S_i}^{L_{S_i,S_j,1,1}}\gamma_{S_i,S_j}^{L_{S_i,S_j,1,1}}=-\varepsilon_{k+1}, \\
         &  \varepsilon_{k+2}'=\varepsilon_k(\gamma_{S_i,S_j}^{L_{S_i,S_j,1,1}})^2=\varepsilon_k.
    \end{align*}
    Similarly, using the previous Lemma, we obtain that  $\varepsilon_l'=\varepsilon_l$ for all $l\neq k,k+1,k+2$.
    The Lemma is proved.
\end{proof}

\begin{definition}
    For $0\leq k \leq m-1$, we inductively define 
\begin{align*}
    \varphi_k:(\varepsilon_1 \mathbb{R}_{>0})\times \cdots \times (\varepsilon_k \mathbb{R}_{>0}) \times  \mathbb{R}_{>0}^{m-k} \rightarrow (\varepsilon_1 \mathbb{R}_{>0})\times \cdots \times  (\varepsilon_{k+1} \mathbb{R}_{>0}) \times \mathbb{R}_{>0}^{m-k-1}.
\end{align*}
For any $(a_1^{(k)},\cdots,a_m^{(k)})\in (\varepsilon_1 \mathbb{R}_{>0})\times \cdots \times (\varepsilon_k \mathbb{R}_{>0}) \times  \mathbb{R}_{>0}^{m-k}$, let $(a_1^{(k+1)},\cdots,a_m^{(k+1)})=\varphi_k(a_1^{(k)},\cdots,a_m^{(k)})$, such that:

(1) For $1\leq l \leq k$, we have $a_l^{(k+1)}=a_l^{(k)}$.

(2) For $l=k+1$, we have $a_{k+1}^{(k+1)}=\varepsilon_{k+1}(a_{k+1}^{(k)})^{-1}$.

(3) For $k+2\leq l\leq m$ and $i_l\neq i_{k+1}$, we have 
\begin{align*}
    a_l^{(k+1)}=a_l^{(k)}( a_{k+1}^{(k)}+\sum_{\substack{k+1<s<l\\i_s=i_{k+1}}} a_s^{(k)} )^{-A_{S_{i_{k+1}},S_{i_l}}}.
\end{align*}

(4) For $k+2\leq l\leq m$ and $i_l= i_{k+1}$, we have 
\begin{align*}
    a_l^{(k+1)}=a_l^{(k)}( a_{k+1}^{(k)}+\sum_{\substack{k+1<s<l\\i_s=i_{k+1}}} a_s^{(k)} )^{-1} ( a_{k+1}^{(k)}+\sum_{\substack{k+1<s\leq l\\i_s=i_{k+1}}} a_s^{(k)} )^{-1}.
\end{align*}
\end{definition}

\begin{lemma}\label{varphikbij}
The map $\varphi_k$ gives a homeomorphism from $(\varepsilon_1 \mathbb{R}_{>0})\times \cdots \times (\varepsilon_k \mathbb{R}_{>0}) \times  \mathbb{R}_{>0}^{m-k}$ to 
\begin{align*}
    \{ (a_1,\cdots,a_m)\in (\varepsilon_1 \mathbb{R}_{>0})\times \cdots \times  (\varepsilon_{k+1} \mathbb{R}_{>0}) \times \mathbb{R}_{>0}^{m-k-1}| \varepsilon_{k+1}a_{k+1} > \sum_{\substack{k+2<s\leq m\\i_s=i_{k+1}}} a_s\}.
\end{align*}
 (If there doesn't exist $k+2 <s\leq m$ such that $i_s=i_{k+1}$, let $\sum_{\substack{k+2<s\leq m\\i_s=i_{k+1}}} a_s=0$).
\end{lemma}

\begin{proof}
    By definition, $\varphi_k$ is continuous.

     Assume $(a_1^{(k+1)},\cdots,a_m^{(k+1)})\in (\varepsilon_1 \mathbb{R}_{>0})\times \cdots \times  (\varepsilon_{k+1} \mathbb{R}_{>0}) \times \mathbb{R}_{>0}^{m-k-1}$ satisfies 
    \begin{align*}
        \varepsilon_{k+1}a_{k+1}^{(k+1)} > \sum_{\substack{k+2<s\leq m\\i_s=i_{k+1}}} a_s^{(k+1)},
    \end{align*}
   
    Let $(a_1^{(k)},\cdots,a_m^{(k)})\in (\varepsilon_1 \mathbb{R}_{>0})\times \cdots \times (\varepsilon_k \mathbb{R}_{>0}) \times  \mathbb{R}_{>0}^{m-k}$ given by:

    (1) For $1\leq l \leq k$, we have $a_l^{(k)}=a_l^{(k+1)}$.

    (2) For $l=k+1$, we have $a_{k+1}^{(k)}=\varepsilon_{k+1}(a_{k+1}^{(k+1)})^{-1}$.

    (3) For $k+2\leq l\leq m$ and $i_l\neq i_{k+1}$, we have 
\begin{align*}
    a_l^{(k)}=a_l^{(k+1)}(\varepsilon_{k+1} a_{k+1}^{(k+1)}-\sum_{\substack{k+1<s<l\\i_s=i_{k+1}}} a_s^{(k+1)} )^{-A_{S_{i_{k+1}},S_{i_l}}}.
\end{align*}

    (4) For $k+2\leq l\leq m$ and $i_l= i_{k+1}$, we have 
\begin{align*}
    a_l^{(k)}=a_l^{(k+1)}(\varepsilon_{k+1} a_{k+1}^{(k+1)}-\sum_{\substack{k+1<s<l\\i_s=i_{k+1}}} a_s^{(k+1)} )^{-1} (\varepsilon_{k+1} a_{k+1}^{(k+1)}-\sum_{\substack{k+1<s\leq l\\i_s=i_{k+1}}} a_s^{(k+1)} )^{-1}.
\end{align*}
It's easy to check that $(a_1^{(k+1)},\cdots,a_m^{(k+1)})=\varphi_k(a_1^{(k)},\cdots,a_m^{(k)})$.
This gives the inverse of $\varphi_k$, which is also continuous. 
Thus $\varphi_k$ is a homeomorphism.
\end{proof}

For any $b_1,\cdots,b_m\in \mathbb{R}_{>0}$, let 
\begin{align*}
    (a_1^{(m)}(b_1,\cdots,b_m),\cdots,a_m^{(m)}(b_1,\cdots,b_m))=(\varepsilon_1 b_1,\cdots,\varepsilon_m b_m).
\end{align*}
For $k<m$, if functions $a_t^{(k+1)}=a_t^{(k+1)}(b_1,\cdots,b_m)$ is defined for $t=1,\cdots,m$, let 
\begin{align*}
    \beta_{k+1}(b_1,\cdots,b_m)= \varepsilon_{k+1}a_{k+1}^{(k+1)} - \sum_{\substack{k+2<s\leq m\\i_s=i_{k+1}}} a_s^{(k+1)}.
\end{align*}
(If there doesn't exist $k+2 <s\leq m$ such that $i_s=i_{k+1}$, let $ \beta_{k+1}(b_1,\cdots,b_m)= \varepsilon_{k+1}a_{k+1}^{(k+1)} $).
Assume that $(b_1,\cdots,b_m)$ satisfies $\beta_{k+1}(b_1,\cdots,b_m)>0$, and then define $a_t^{(k)}=a_t^{(k)}(b_1,\cdots,b_m)$, $t=1,\cdots,m$, such that 
\begin{align*}
    (a_1^{(k)},\cdots,a_m^{(k)})=\varphi_{k}^{-1}(a_1^{(k+1)},\cdots,a_m^{(k+1)}).
\end{align*}
Inductively, we define functions $a_t^{(l)}$ and $\beta_k$ for $1\leq t,k\leq m$ and $0\leq l \leq m $.

\begin{lemma}\label{determine}
    For any $t>k$, $a_t^{(k)}=a_t^{(k)}(b_1,\cdots,b_m)$ is in $\mathbb{R}_{>0}$, and is determined by $b_{k+1},\cdots,b_t$.
\end{lemma}

\begin{proof}
    It follows from definition that $a_t^{(k)}>0$ for all $t>k$.
    Note that $a_k^{(k)}=a_k^{(k+1)}=\cdots=a_k^{(m)}=\varepsilon_k b_k$ and $a_k^{(k-1)}=b_k^{-1}$, for all $k$.

    We prove the Lemma by induction on $k$.
    Firstly, $a_m^{(m-1)}=b_m^{-1}$ is determined by $b_m$.
    Assume that for any $t>k+1$, $a_t^{(k+1)}$ is determined by $b_{k+2},\cdots,b_t$.
    We have $a_{k+1}^{(k)}=b_{k+1}^{-1}$ is determined by $b_{k+1}$. For $k+2 \leq t \leq m$, if $i_t\neq i_{k+1}$, we have 
\begin{align*}
    a_t^{(k)}=a_t^{(k+1)}(\varepsilon_{k+1} a_{k+1}^{(k+1)}-\sum_{\substack{k+1<s<t\\i_s=i_{k+1}}} a_s^{(k+1)} )^{-A_{S_{i_{k+1}},S_{i_t}}},
\end{align*}
where $a_t^{(k+1)}$ is determined by $b_{k+2},\cdots,b_t$, $a_{k+1}^{(k+1)}$ is determined by $b_{k+1}$, and $a_s^{(k+1)}$ is determined by $b_{k+2},\cdots,b_s$ with $s<t$.
Thus $a_t^{(k)}$ is determined by $b_{k+1},\cdots,b_t$.
Similar for the case $i_t= i_{k+1}$.
By induction, the Lemma is proved.
\end{proof}

\begin{lemma}
    The $m$ inequalities $\beta_k(b_1,\cdots,b_m)>0$, for $k=1,\cdots,m$, are nontrivial and independent, which means none of them can be deduced from any collection of the others.
\end{lemma}

\begin{proof}
    By definition,
    \begin{align*}
    \beta_{k}= \varepsilon_{k}a_{k}^{(k)} - \sum_{\substack{k+1<s\leq m\\i_s=i_{k}}} a_s^{(k)} = b_{k}-\sum_{\substack{k+1<s\leq m\\i_s=i_{k}}} a_s^{(k)} .
    \end{align*}
    By Lemma \ref{determine}, the sum $\sum_{\substack{k+1<s\leq m\\i_s=i_{k}}} a_s^{(k)}$ is determined by $b_{k+1},\cdots,b_l$ (or is zero), where $l\leq m$ is largest such that $i_l=i_k$.
    So the coefficient of $b_k$ in the function $\beta_k$ is 1, and $\beta_k$ is a rational function of $b_k,b_{k+1},\cdots,b_l$. 
    The result follows.
\end{proof}

\begin{definition}
    For an arbitrarily fixed reduced expression $(i_1,\cdots,i_m)$ of $\overline{n}_0$ and use the notations as above, define a region 
    \begin{align*}
        \Omega_{i_1,\cdots,i_m} = \{ (b_1,\cdots,b_m)\in \mathbb{R}^m_{>0} | \beta_k(b_1,\cdots,b_m)>0, \forall k=1,\cdots,m  \}.
    \end{align*}
\end{definition}

When fixing a complete section of $\mathcal{R}$, we actually give an orientation of the corresponding graph and get a quiver $Q$. 
Let $Q_0=\{1,\cdots,n\}$ be the set of vertices of $Q$.
For each $i\in Q_0$, let $\sigma_i$ be the action of reversing all the arrows at $i$.
Without loss of generality, we may assume 1 is a source in $Q$, 2 is a source in $\sigma_1(Q)$, and so on. Finally, assume $n$ is a source in $\sigma_{n-1}\cdots\sigma_1(Q)$.
Note that $\sigma_n \cdots \sigma_1(Q)=Q$, since each arrow $i\rightarrow j$ reverse twice (when apply $\sigma_i$ and $\sigma_j$).
Let $k$ be an algebraically closed field.
Let
\begin{align*}
    I_1,\cdots,I_n,\tau I_1,\cdots,\tau I_n, \tau^2 I_1,\cdots
\end{align*}
be an ordering of indecomposable objects in $\mathcal{R}$, where $\tau$ is the AR-translation functor.
We consider its intersection with $\operatorname{ind}\mathcal{B}$, which leads to an ordering of the indecomposable $kQ$-modules.
Following \cite{ORT} (note that they use left $kQ$-module while we use right $kQ$-module), we can get a reduced expression for $\overline{n}_0$ by replacing each $\tau^jI_i$ in the sequence by $i$. This reduced expression is called the leftmost word for $\overline{n}_0$ (in $Q$).
In this case, we may also denote the corresponding region by  $\Omega_Q$.

\begin{example}
    We consider the quiver $Q$:
    \begin{align*}
        1\rightarrow 2 \rightarrow 3
    \end{align*}
    and the leftmost word $(1,2,3,1,2,1)$ for $\overline{n}_0$. 
    We have  
    \begin{align*}
        &\varepsilon_1=1, \\
        & \varepsilon_2=\eta_{I_1,I_2}=\gamma_{\tau I_1,I_1}^{I_2}=1,\\
        &\varepsilon_3 = \eta_{I_1,I_3}\eta_{\tau I_1,\tau I_2}=\gamma_{\tau I_2,I_1}^{I_3}\gamma_{\tau^2 I_1,\tau I_1}^{\tau I_2}=1,\\
        &\varepsilon_4 = \eta_{I_1,\tau I_1}\eta_{\tau I_1,I_2}\eta_{\tau^2 I_1,I_1}=(\gamma_{I_1,\tau I_1}^{I_2} )^2 \cdot 1=1,\\
        &\varepsilon_5 = \eta_{I_1,\tau I_2}\eta_{\tau I_1,I_3}\eta_{\tau^2 I_1,I_3}\eta_{I_1,I_2}=\gamma_{I_1,\tau I_2}^{I_3}\cdot 1 \cdot \gamma_{I_2,\tau^2 I_1}^{I_3} \gamma_{\tau I_1,I_1}^{I_2} = (-1)^2 =1,\\
        &\varepsilon_6 = \eta_{I_1,\tau^2 I_1} \eta_{\tau I_1,\tau^2 I_1} \eta_{\tau^2 I_1,\tau I_2} \eta_{I_1,\tau I_1} \eta_{\tau I_1,I_2} = 1 \cdot (\gamma_{\tau I_1,\tau^2 I_1}^{\tau I_2})^2 (\gamma_{I_1,\tau I_1}^{I_2})^2 = 1.
    \end{align*}
    For $b_1,\cdots,b_6 \in \mathbb{R}_{>0}$, we can calculate inductively:
    \begin{align*}
        (a_1^{(6)},\cdots,a_6^{(6)})&=(b_1,\cdots,b_6),\\
        \beta_6&=b_6>0,\\
        (a_1^{(5)},\cdots,a_6^{(5)})&=(b_1,\cdots,b_5,b_6^{-1}),\\
        \beta_5&=b_5>0,\\
        (a_1^{(4)},\cdots,a_6^{(4)})&=(b_1,\cdots,b_4,b_5^{-1},b_5b_6^{-1}),\\
        \beta_4&=b_4-b_5b_6^{-1}>0,\\
        (a_1^{(3)},\cdots,a_6^{(3)})&=(b_1,b_2,b_3,b_4^{-1},b_4b_5^{-1},\frac{b_4^{-2}b_5b_6^{-1}}{1-b_4^{-1}b_5b_6^{-1}}),\\
        \beta_3&=b_3>0,\\
        (a_1^{(2)},\cdots,a_6^{(2)})&=(b_1,b_2,b_3^{-1},b_4^{-1},b_3b_4b_5^{-1},\frac{b_4^{-2}b_5b_6^{-1}}{1-b_4^{-1}b_5b_6^{-1}}),\\
        \beta_2&=b_2-b_3b_4b_5^{-1}>0,\\
        (a_1^{(1)},\cdots,a_6^{(1)})&=(b_1,b_2^{-1},b_2b_3^{-1},b_2b_4^{-1},\frac{b_2^{-2}b_3b_4b_5^{-1}}{1-b_2^{-1}b_3b_4b_5^{-1}},\frac{b_2b_4^{-2}b_5b_6^{-1}(1-b_2^{-1}b_3b_4b_5^{-1})}{1-b_4^{-1}b_5b_6^{-1}}),\\
        \beta_1&=b_1-\frac{b_2b_6-b_3}{b_4b_6-b_5}>0,\\
        (a_1^{(0)},\cdots,a_6^{(0)})&=(b_1^{-1},b_1b_2^{-1},b_2b_3^{-1},\frac{b_1^{-2}b_2b_4^{-1}}{1-b_1^{-1}b_2b_4^{-1}},\frac{b_1b_2^{-2}b_3b_4b_5^{-1}(1-b_1^{-1}b_2b_4^{-1})}{1-b_2^{-1}b_3b_4b_5^{-1}},\\
        &\frac{b_2b_5-b_3b_4}{(b_1b_4-b_2)(b_1b_4b_6-b_1b_5-b_2b_6+b_3)}).
    \end{align*}
    The region is 
    \begin{align*}
        \Omega_Q=\{(b_1,\cdots,b_6)\in \mathbb{R}_{>0}^6\mid
        \substack{b_1b_4b_6-b_1b_5-b_2b_6+b_3>0,\\b_2b_5-b_3b_4>0,\\b_4b_6-b_5>0.}  \}.
    \end{align*}
\end{example}

Let $Q'=\sigma_1(Q)$ and let $I_1',\cdots,I_n'$ be the injective $kQ'$-modules.
The BGP reflection functor $R(S_1^+)$ maps $\tau I_1$ to $I_1'$, and maps $I_k$ to $I_k'$ for any $k\neq 1$.
By the assumption on $Q$, we have 2 is a source in $Q'$, 3 is a source in $\sigma_2(Q')$, and so on.
In particular, 1 is a source in $Q=\sigma_n \cdots \sigma_1(Q)=\sigma_n \cdots \sigma_2(Q')$.
Thus the ordering of the indecomposable $kQ'$-modules is 
\begin{align*}
      I_2',\cdots,I_n',I_1',\tau I_2',\cdots.
\end{align*}

Suppose for $\overline{n}_0$, the leftmost word in $Q$ is $s_{i_1} s_{i_2}\cdots s_{i_m}$, and then the leftmost word in $Q'$ is $s_{i_2}\cdots s_{i_m}s_{i_{m+1}}$.
For $2\leq k \leq m$, define 
\begin{align*}
    \varepsilon_k' = \eta_{S_{i_{2}},\omega_{S_{i_{2}}}\cdots \omega_{S_{i_{k}}}(S_{i_{k+1}})}  \cdots \eta_{S_{i_{k}},\omega_{S_{i_{k}}}(S_{i_{k+1}})},
\end{align*}
and let $\varepsilon_1'=1$.
Then for $2\leq k \leq m$, we have $\varepsilon_k=\eta_{S_{i_1},\omega_{S_{i_1}}\cdots\omega_{S_{i_{k-1}}}(S_{i_k})}\varepsilon_{k-1}'$.

For any $b_1,\cdots,b_{m+1}\in \mathbb{R}_{>0}$, let 
\begin{align*}
    (a_1^{(m)},\cdots,a_m^{(m)})=(\varepsilon_1b_1,\cdots,\varepsilon_mb_m), \quad ({a_1'}^{(m)},\cdots,{a_m'}^{(m)})=(\varepsilon_1'b_2,\cdots,\varepsilon_m'b_{m+1}).
\end{align*}
For $1\leq k,l\leq m$, let $\beta_k=\beta_k(b_1,\cdots,b_m),a_l^{(k)}=a_l^{(k)}(b_1,\cdots,b_m)$ as defined previously.
Similarly, we have $\varphi_k'$, $\beta_k'=\beta_k'(b_2,\cdots,b_{m+1})$ and ${a_l'}^{(k)}={a_l'}^{(k)}(b_2,\cdots,b_{m+1})$, for the reduced expression $(i_2,\cdots,i_{m+1})$.
Denote $j=i_{m+1}$.

\begin{lemma}
    For any $1\leq k \leq m$, assume that $\beta_{k+1},\cdots,\beta_m,\beta_k',\cdots,\beta_m'>0$.
    Then for any $k+1 \leq l < m+1$, we have ${a_{l-1}'}^{(k-1)}=a_l^{(k)}$.
\end{lemma}

\begin{proof}
    The assumption guarantees that ${a_{l-1}'}^{(k-1)}$ and $a_l^{(k)}$ are well-defined. Then the Lemma follows immediately from the definitions.
\end{proof}

\begin{corollary}
    For any $2\leq k \leq m$, assume that $\beta_{k+1},\cdots,\beta_m,\beta_k',\cdots,\beta_m'>0$.
    If $i_k\neq j$, then $\beta_{k-1}'=\beta_k$.
\end{corollary}

\begin{proof}
    We have 
    \begin{align*}
        \beta_{k-1}'
        = b_k-\sum_{\substack{k<s\leq m \\ i_{s+1}=i_k}} {a_s'}^{(k-1)}
        = b_k-\sum_{\substack{k+1< t \leq m+1 \\ i_t=i_k}} {a_{t-1}'}^{(k-1)}
        = b_k-\sum_{\substack{k+1< t \leq m \\ i_t=i_k}} a_{t}^{(k)} = \beta_k,
    \end{align*}
    where the third equality follows from $i_k\neq j=i_{m+1}$ and the previous Lemma.
\end{proof}

\begin{proposition}\label{Omega}
    There is a homeomorphism between the regions $\Omega_Q$ and $\Omega_{Q'}$.
\end{proposition}

\begin{proof}
    Recall that $\Omega_{Q}=\Omega_{i_1,\cdots,i_m}$, $\Omega_{Q'}=\Omega_{i_2,\cdots,i_{m+1}}$.
    It only remains to consider those $k$ such that $i_k=j$.

    If there doesn't exist any $2\leq k\leq m$ such that $i_k=j$, then we must have $j=i_1$, since $1,\cdots,n$ all appear in a reduced expression of $\overline{n}_0$.
    In this case, $\beta_{k-1}'=\beta_k$ for all $2\leq k\leq m$, $\beta_1=b_1$, $\beta_m'=b_{m+1}$.
    There is obviously a bijection between $\Omega_{i_1,\cdots,i_m}$ and $\Omega_{i_2,\cdots,i_{m+1}}$.

    Now assume $2\leq k_t < \cdots <k_1 \leq m$ are all the numbers such that $j=i_{k_1}=\cdots=i_{k_t}$.
    Let $\beta_s>0$ or $\beta_s'>0$ inductively to assure $ {a_m'}^{(l)}$ is well-defined.
    For $k_1-1\leq l<m$, we have 
    \begin{align*}
        {a_m'}^{(l)}=b_{m+1}^{-1}\beta_m^{-A_{S_{i_m},S_j}}\cdots \beta_{l+2}^{-A_{S_{i_{l+2}},S_j}}.
    \end{align*}
    Then
    \begin{align*}
    \beta_{k_1-1}'=b_{k_1}-{a_m'}^{(k_1-1)}=b_{k_1}-b_{m+1}^{-1}\beta_m^{-A_{S_{i_m},S_j}}\cdots \beta_{k_1+1}^{-A_{S_{i_{k_1+1}},S_j}},
    \end{align*}
    while $\beta_{k_1}=b_{k_1}$.
    For $l=k_1-2$, we have 
    \begin{align*}
        {a_m'}^{(k_1-2)}={a_m'}^{(k_1-1)} b_{k_1}^{-1}(b_{k_1}-{a_m'}^{(k_1-1)})^{-1}.
    \end{align*}
    For $k_2-1\leq l< k_1-2$, we have 
    \begin{align*}
        {a_m'}^{(l)}={a_m'}^{(k_1-2)}\beta_{k_1-1}^{-A_{S_{i_{k_1-1}},S_j}}\cdots \beta_{l+2}^{-A_{S_{i_{l+2}},S_j}},
    \end{align*}
    and so on. 
    We can calculate $ {a_m'}^{(l)}$ inductively, and have
    \begin{align*}
        \beta_{k_s-1}'=\beta_{k_s}- {a_m'}^{(k_s-1)},
    \end{align*}
    for $1\leq s \leq t$.
    Again $\beta_m'=b_{m+1}$, while $\beta_1=b_1-\sum_{\substack{2< s \leq m \\ i_s=i_1}} a_{s}^{(1)}$.
    Thus there is a bijection between the regions $\Omega_{i_1,\cdots,i_m}$ and $\Omega_{i_2,\cdots,i_{m+1}}$.
    One can easily see that the bijection is continuous, and so is its inverse.
    The proposition is proved.
\end{proof}

For any two complete sections, and consequently two complete hereditary subcategories of $\mathcal{R}$, there is a sequence of BGP reflection functors maps one to another. 
Repeatedly apply Proposition \ref{Omega}, we get a bijection between the two corresponding regions.
Thus we may regard these regions $\Omega_Q$ as obtained from $\mathcal{R}$, and call them fundamental regions from the root category.

\begin{remark}
    Although the definition of a region $\Omega_{i_1,\cdots,i_m}$ is valid for any reduced expression $(i_1,\cdots,i_m)$ of $\overline{n}_0$, we will only use part of them, in order to have a better connection with quiver representations. 
    (There indeed exist reduced expressions for $\overline{n}_0$ that are not leftmost words in all orientations of  quivers. For example, type $A_3$ and reduced expression $(1,2,3,2,1,2)$). 
    But the results presented later in this paper are applicable to any reduced expression.
\end{remark}

\subsection{Description of $\mathbf{U}^+_{>0}$, $\mathbf{U}^-_{>0}$ and $\mathbf{G}_{>0}$}

In this subsection, we will use the fundamental regions of $\mathcal{R}$ to characterize the size of the monoids of  totally positive elements in $\mathbf{G}=\mathbf{G}(\mathcal{R},\mathbb{R})$.

We fix a complete section of $\mathcal{R}$, and then get a complete hereditary subcategory $\mathcal{B}$ of $\mathcal{R}$ and a quiver $Q$. 
The set of vertices $Q_0=\{1,\cdots,n\}$, and we assume that $i$ is a source in $\sigma_{i-1}\cdots\sigma_1(Q)$ for $i=1,\cdots,n$.
Consider the ordering 
\begin{align*}
  \{  I_1,\cdots,I_n,\tau I_1,\cdots ,\tau I_n, \tau^2 I_1, \cdots \}\bigcap \operatorname{ind}\mathcal{B}
\end{align*}
of indecomposable $\mathbb{C}Q$-modules.
Assume that there are a total of $m$ indecomposable modules in the ordering, and denote the indecomposable module at the $k$-th position in the ordering as $\tau^{j_k}I_{i_k}$, for $1\leq k\leq m$.
Note that $(i_1,\cdots,i_m)$ is the leftmost word for $\overline{n}_0$ in $Q$.

\begin{lemma}
    For any $1\leq k\leq m$, $i_k$ is a source in $\sigma_{i_{k-1}}\cdots\sigma_{i_1}(Q)$.
\end{lemma}

\begin{proof}
    Note that $j\geq 1$ is minimal such that $\tau^j I_i\notin \operatorname{ind}\mathcal{B}$ if and only if $\tau^{j-1}I_i$ is projective as $\mathbb{C}Q$-module.
    For $1\leq k\leq n$, $i_k$ is a source in $\sigma_{i_{k-1}}\cdots\sigma_{i_1}(Q)$ by the assumption on $Q$.

    If $\tau I_i\notin \operatorname{ind}\mathcal{B}$ for some $i$ (and will be dropped from the ordering), then $I_i$ is both projective and injective, which implies $i$ is a sink in $Q$.
    Then there is no arrow between $i$ and any of the vertices $i+1,\cdots,n$.
    Thus $l$ is a source in $\sigma_{l-1}\cdots\sigma_{i+1}\sigma_{i-1}\cdots\sigma_1(Q)=\sigma_{l-1}\cdots\sigma_{i+1}\sigma_{i-1}\cdots\sigma_1\sigma_{n}\cdots\sigma_1(Q)$, for any $i+1\leq l\leq n$.
    This shows that, for all $k$ such that $j_k=1$, we have $i_k$ is a source in $\sigma_{i_{k-1}}\cdots\sigma_{i_1}(Q)$.

    Assume that $\tau I_{p_1},\cdots,\tau I_{p_s}$ are all terms in $\{\tau I_1,\cdots,\tau I_n\}$ that are not in $\operatorname{ind}\mathcal{B}$, with $1\leq p_1<\cdots<p_s\leq n$.
    If $\tau^2 I_j\notin \operatorname{ind}\mathcal{B}$ for some $j$, then $\tau I_j$ is projective.
    Consider the quiver $Q_1=\sigma_n \cdots \hat{\sigma}_{p_s}\cdots \hat{\sigma}_{p_1}\cdots \sigma_1(Q)$, which means $\sigma_{p_1},\cdots,\sigma_{p_s}$ are omitted from $\sigma_n\cdots\sigma_1$.
    Then in $Q_1$, the vertices $p_1,\cdots,p_s$ become sources.
    The composition of BGP reflection functors
    \begin{align*}
        R(S_n^+)\cdots \hat{R}(S_{p_s}^+)\cdots \hat{R}(S_{p_1}^+)\cdots R(S_1^+):\operatorname{mod}\mathbb{C}Q\rightarrow \operatorname{mod}\mathbb{C}Q_1,
    \end{align*}
    which omit functors $\hat{R}(S_{p_1}^+),\cdots,\hat{R}(S_{p_s}^+)$, maps $\tau I_j$ to the injective module $I_j'$ corresponding to vertex $j$ in $\operatorname{mod}\mathbb{C}Q_1$.
    Since $\tau I_j$ is projective, we have $I_j'$ is both injective and projective, which means $j$ is a sink in $Q_1$.
    Similar to the preceding argument, for $j+1\leq l\leq n$, $l$ is a source in $\sigma_{l-1}\cdots \hat{\sigma}_j \cdots \hat{\sigma}_{p_t} \cdots \sigma_1(Q_1)$, which omit $\sigma_j$ and all the $\sigma_{p_t}$ such that $p_t<l$. 
    Since $Q_1=\sigma_n \cdots \hat{\sigma}_{p_s}\cdots \hat{\sigma}_{p_1}\cdots \sigma_1\sigma_n\cdots\sigma_1(Q)$, we can deduce that for all $k$ such that $j_k=2$, $i_k$ is a source in $\sigma_{i_{k-1}}\cdots\sigma_{i_1}(Q)$.

    By similar reasoning, we can prove that for all $1\leq k\leq m$, $i_k$ is a source in $\sigma_{i_{k-1}}\cdots\sigma_{i_1}(Q)$.
\end{proof}

\begin{lemma}\label{tauIS}
    For any $2\leq k\leq m$, the composition of BGP reflection functors $R(S_{i_{k-1}}^+)\cdots R(S_{i_1}^+)$ maps $\tau^{j_k}I_{i_k}$ to $S_{i_k}$.
\end{lemma}

\begin{proof}
    By the previous Lemma, the composition of BGP reflection functors
    \begin{align*}
        R(S_{i_{k-1}}^+)\cdots R(S_{i_1}^+)
    \end{align*}
     is well-defined.
    Note that for any source $i$ in some quiver $Q'$, the BGP reflection functor $R(S_i^+)$ maps $\tau I_i$ in $\operatorname{mod}\mathbb{C}Q'$ to the injective module $I_i'$ in $\operatorname{mod}\mathbb{C}\sigma_i(Q')$, and maps $I_j$ in $\operatorname{mod}\mathbb{C}Q'$ to the injective module $I_j'$ in $\operatorname{mod}\mathbb{C}\sigma_i(Q')$, for any $j\neq i$.
    Thus we have 
    \begin{align*}
        j_k=\#\{ 1\leq p\leq k-1 \mid i_p=i_k \},
    \end{align*}
    where $\#$ denote the number of elements in a set.

    From this fact, it follows that the functor $R(S_{i_{k-1}}^+)\cdots R(S_{i_1}^+)$ maps $\tau^{j_k}I_{i_k}$ to the injective module corresponding to $i_k$ of the quiver $\sigma_{i_{k-1}}\cdots\sigma_{i_1}(Q)$.
    Since $i_k$ is a source in $\sigma_{i_{k-1}}\cdots\sigma_{i_1}(Q)$, the injective module corresponding to $i_k$ is simple, that is, $S_{i_k}$.
\end{proof}

By Lemma \ref{tauIS}, we have for $1\leq k\leq m$
\begin{align*}
    n_{[S_{i_{k-1}}]}\cdots n_{[S_{i_1}]}u_{\tau^{j_k}I_{i_k}}=\varepsilon_k u_{S_{i_k}},
\end{align*}
where $\varepsilon_1,\cdots,\varepsilon_m$ are exactly the Tits signs for $(i_1,\cdots,i_m)$.

\begin{definition}
    Let $\mathfrak{U}^+_{>0}$ be a subset of $\mathbf{U}^+$ consisting of the following elements:
    \begin{align*}
        \mathfrak{U}^+_{>0}=\{ E_{[\tau^{j_1}I_{i_1}]}(\varepsilon_1b_1)\cdots  E_{[\tau^{j_m}I_{i_m}]}(\varepsilon_mb_m) \mid (b_1,\cdots,b_m)\in \Omega_Q \}.
    \end{align*}
\end{definition}

We will compare $\mathfrak{U}^+_{>0}$ and the submonoid $\mathbf{U}^+_{>0}$ of $\mathbf{G}=\mathbf{G}(\mathcal{R},\mathbb{R})$.

Let $\mathfrak{B}$ be the set of all Borel subgroups of $\mathbf{G}$.
For any $x\in \mathbf{G}$ and any Borel subgroup $\mathbf{B}\in \mathfrak{B}$, denote $x\mathbf{B}x^{-1}$ by $x\cdot \mathbf{B}$.

For $0\leq k\leq m$, define a map $f_k:(\mathbb{R}^{\times})^m \rightarrow \mathfrak{B}$ by 
\begin{align*}
    &f_k(a_1,\cdots,a_m)\\
    =&E_{[\tau^{j_1}I_{i_1}]}(a_1) \cdots E_{[\tau^{j_k}I_{i_k}]}(a_k) n_{i_1}^{-1}\cdots n_{i_k}^{-1}E_{[TS_{i_{k+1}}]}(-a_{k+1}) \cdots E_{[TS_{i_m}]}(-a_m) \cdot \mathbf{B}^+.
\end{align*}
In particular, we have 
\begin{align*}
    f_0(a_1,\cdots,a_m)&=E_{[TS_{i_1}]}(-a_1)\cdots E_{[TS_{i_m}]}(-a_m)\cdot \mathbf{B}^+,\\
    f_m(a_1,\cdots,a_m)&=E_{[\tau^{j_1}I_{i_1}]}(a_1) \cdots E_{[\tau^{j_m}I_{i_m}]}(a_m) n_{i_1}^{-1}\cdots n_{i_m}^{-1} \cdot \mathbf{B}^+\\
    &=E_{[\tau^{j_1}I_{i_1}]}(a_1) \cdots E_{[\tau^{j_m}I_{i_m}]}(a_m)  \cdot \mathbf{B}^-.
\end{align*}

\begin{lemma}
    For any $X\in \operatorname{ind}\mathcal{R}$ and $t\neq 0$, we have 
    \begin{align*}
        &E_{[TX]}(-t)=E_{[X]}(t^{-1})n_{[X]}^{-1}h_{[X]}(t)E_{[X]}(t^{-1}),\\
        &E_{[X]}(t)=E_{[TX]}(-t^{-1}) n_{[X]} h_{[X]}(t^{-1})E_{[TX]}(-t^{-1}).
    \end{align*}
\end{lemma}

\begin{proof}
    This Lemma immediately follows from the fact that 
    \begin{align*}
        E_{[X]}(t)E_{[TX]}(t^{-1})E_{[X]}(t)=n_{[X]}(t)=h_{[X]}(t)n_{[X]}
    \end{align*}
    and Lemma \ref{conjugate}.
\end{proof}

\begin{proposition}\label{fk}
   For each $0\leq k\leq m-1$ and $(a_1,\cdots,a_m)\in (\varepsilon_1\mathbb{R}_{>0})\times \cdots \times (\varepsilon_k\mathbb{R}_{>0})\times \mathbb{R}_{>0}^{m-k}$, we have 
    \begin{align*}
        f_k(a_1,\cdots,a_m)=f_{k+1}\varphi_k(a_1,\cdots,a_m).
    \end{align*}
\end{proposition}

\begin{proof}
    Denote $h_{[S_i]}(t)$ by $h_i(t)$.
For $(a_1,\cdots,a_m)\in (\mathbb{R}^{\times})^m$, we have 
\begin{align*}
    &f_k(a_1,\cdots,a_m)\\
    =&E_{[\tau^{j_1}I_{i_1}]}(a_1) \cdots E_{[\tau^{j_k}I_{i_k}]}(a_k) n_{i_1}^{-1}\cdots n_{i_k}^{-1} E_{[S_{i_{k+1}}]}(a_{k+1}^{-1}) n_{i_{k+1}}^{-1} h_{i_{k+1}}(a_{k+1}) E_{[S_{i_{k+1}}]}(a_{k+1}^{-1})\times \\
    & \times  E_{[TS_{i_{k+2}}]}(-a_{k+2}) \cdots E_{[TS_{i_m}]}(-a_m) \cdot \mathbf{B}^+\\
    =&E_{[\tau^{j_1}I_{i_1}]}(a_1) \cdots E_{[\tau^{j_k}I_{i_k}]}(a_k) E_{[\tau^{j_{k+1}}I_{i_{k+1}}]}(\varepsilon_{k+1}a_{k+1}^{-1}) n_{i_1}^{-1}\cdots n_{i_{k+1}}^{-1}\times \\ 
    & \times h_{i_{k+1}}(a_{k+1}) E_{[S_{i_{k+1}}]}(a_{k+1}^{-1}) E_{[TS_{i_{k+2}}]}(-a_{k+2}) \cdots E_{[TS_{i_m}]}(-a_m) \cdot \mathbf{B}^+.
\end{align*}

    For $k+1\leq s \leq m$, define a map $f_s':\mathbb{R}_{>0}^{m-k}\rightarrow \mathfrak{B}$ by 
    \begin{align*}
        &f_s'(c_{k+1},\cdots,c_m)\\
        =&E_{[TS_{i_{k+2}}]}(-c_{k+2})\cdots E_{[TS_{i_{s}}]}(-c_s) E_{[S_{i_{k+1}}]}(c_{k+1}) E_{[TS_{i_{s+1}}]}(-c_{s+1})\cdots E_{[TS_{i_m}]}(-c_m)\cdot \mathbf{B}^+.
    \end{align*}

    In particular, we have 
    \begin{align*}
        f_{k+1}'(c_{k+1},\cdots,c_m)&=E_{[S_{i_{k+1}}]}(c_{k+1}) E_{[TS_{i_{k+2}}]}(-c_{k+2})\cdots E_{[TS_{i_{m}}]}(-c_m)\cdot \mathbf{B}^+,\\
        f_m'(c_{k+1},\cdots,c_m)&=E_{[TS_{i_{k+2}}]}(-c_{k+2})\cdots E_{[TS_{i_{m}}]}(-c_m)E_{[S_{i_{k+1}}]}(c_{k+1}) \cdot \mathbf{B}^+\\
        &= E_{[TS_{i_{k+2}}]}(-c_{k+2})\cdots E_{[TS_{i_{m}}]}(-c_m) \cdot \mathbf{B}^+.
    \end{align*}

    If $i_{k+1}\neq i_{s+1}$, then $f_s'=f_{s+1}'$.

    If $i_{k+1}=i_{s+1}$, then 
    \begin{align*}
        &f_s'(c_{k+1},\cdots,c_m)\\
        &=E_{[TS_{i_{k+2}}]}(-c_{k+2})\cdots E_{[TS_{i_{s}}]}(-c_s) E_{[TS_{i_{s+1}}]}(-\frac{c_{s+1}}{1+c_{k+1}c_{s+1}}) E_{[S_{i_{k+1}}]}((1+c_{k+1}c_{s+1})c_{k+1}) \times \\
        &\times h_{i_{k+1}}(1+c_{k+1}c_{s+1}) E_{[TS_{i_{s+2}}]}(-c_{s+2})\cdots E_{[TS_{i_m}]}(-c_m)\cdot \mathbf{B}^+\\
        &=E_{[TS_{i_{k+2}}]}(-c_{k+2})\cdots E_{[TS_{i_{s}}]}(-c_s) E_{[TS_{i_{s+1}}]}(-\frac{c_{s+1}}{1+c_{k+1}c_{s+1}}) E_{[S_{i_{k+1}}]}((1+c_{k+1}c_{s+1})c_{k+1})\times \\
        &\times E_{[TS_{i_{s+2}}]}(-c_{s+2}(1+c_{k+1}c_{s+1})^{-A_{S_{i_{k+1}},S_{i_{s+2}}}})\cdots E_{[TS_{i_m}]}(-c_m(1+c_{k+1}c_{s+1})^{-A_{S_{i_{k+1}},S_{i_m}}})\cdot \mathbf{B}^+.
    \end{align*}

    Thus we obtain a map $g_s':\mathbb{R}_{>0}^{m-k}\rightarrow \mathbb{R}_{>0}^{m-k}$ such that $f_s'=f_{s+1}'g_s'$.
    For $s$ such that $i_{k+1}\neq i_{s+1}$, let $g_s'=\operatorname{id}$.
    Using these maps, we have 
    \begin{align*}
        f_{k+1}'=f_{k+2}'g_{k+1}'=\cdots=f_m'g_{m-1}'\cdots g_{k+1}'.
    \end{align*}
    Let $g_k=g_{m-1}'\cdots g_{k+1}':\mathbb{R}_{>0}^{m-k}\rightarrow \mathbb{R}_{>0}^{m-k}$ be the composition.

    Now for $(a_1,\cdots,a_m)\in (\varepsilon_1\mathbb{R}_{>0})\times \cdots \times (\varepsilon_k\mathbb{R}_{>0})\times \mathbb{R}_{>0}^{m-k}$, let
    \begin{align*}
       (d_{k+1},\cdots,d_m)=g_k(a_{k+1}^{-1},a_{k+2},\cdots,a_m) .
    \end{align*}
     Then 
    \begin{align*}
    &f_k(a_1,\cdots,a_m)\\
    =&E_{[\tau^{j_1}I_{i_1}]}(a_1) \cdots E_{[\tau^{j_k}I_{i_k}]}(a_k) E_{[\tau^{j_{k+1}}I_{i_{k+1}}]}(\varepsilon_{k+1}a_{k+1}^{-1}) n_{i_1}^{-1}\cdots n_{i_{k+1}}^{-1}\times \\ 
    & \times h_{i_{k+1}}(a_{k+1})  E_{[TS_{i_{k+2}}]}(-d_{k+2}) \cdots E_{[TS_{i_m}]}(-d_m) \cdot \mathbf{B}^+\\
    =&E_{[\tau^{j_1}I_{i_1}]}(a_1) \cdots E_{[\tau^{j_k}I_{i_k}]}(a_k) E_{[\tau^{j_{k+1}}I_{i_{k+1}}]}(\varepsilon_{k+1}a_{k+1}^{-1}) n_{i_1}^{-1}\cdots n_{i_{k+1}}^{-1}\times \\ 
     & \times  E_{[TS_{i_{k+2}}]}(-d_{k+2}a_{k+1}^{-A_{S_{i_{k+1}},S_{i_{k+2}}}}) \cdots E_{[TS_{i_m}]}(-d_ma_{k+1}^{-A_{S_{i_{k+1}},S_{i_m}}}) \cdot \mathbf{B}^+.
    \end{align*}
    That is 
    \begin{align*}
        f_k(a_1,\cdots,a_m)&=f_{k+1}(a_1,\cdots,a_k,\varepsilon_{k+1}a_{k+1}^{-1},d_{k+2}a_{k+1}^{-A_{S_{i_{k+1}},S_{i_{k+2}}}},\cdots,d_ma_{k+1}^{-A_{S_{i_{k+1}},S_{i_m}}})\\
        &=f_{k+1}\varphi_k(a_1,\cdots,a_m),
    \end{align*}
    where the last equality follows from definitions and inductive calculations. 
    The proposition is proved.
\end{proof}

\begin{theorem}\label{u=u}
    The subsets 
    $\mathfrak{U}^+_{>0}=\mathbf{U}^+_{>0}$.
\end{theorem}

\begin{proof}
    For any $(b_1,\cdots,b_m)\in \Omega_Q$, 
\begin{align*}
    (a_1^{(k)},\cdots,a_m^{(k)})= \varphi_k^{-1}\cdots \varphi_{m-1}^{-1}(\varepsilon_1b_1,\cdots,\varepsilon_mb_m)\in (\varepsilon_1\mathbb{R}_{>0})\times \cdots \times (\varepsilon_k\mathbb{R}_{>0})\times \mathbb{R}_{>0}^{m-k}
\end{align*}
can be inductively defined.
Using Proposition \ref{fk}, we have
\begin{align*}
   f_k(a_1^{(k)},\cdots,a_m^{(k)})=f_{k+1}(a_1^{(k+1)},\cdots,a_m^{(k+1)}) 
\end{align*}
 for any $k$.
Thus 
\begin{align*}
    E_{[\tau^{j_1}I_{i_1}]}(\varepsilon_1b_1) \cdots E_{[\tau^{j_m}I_{i_m}]}(\varepsilon_mb_m)  \cdot \mathbf{B}^-    = E_{[TS_{i_1}]}(-a_1^{(0)})\cdots E_{[TS_{i_m}]}(-a_m^{(0)})\cdot \mathbf{B}^+.
\end{align*}
    Since $(a_1^{(0)},\cdots,a_m^{(0)})\in \mathbb{R}^m_{>0}$, we have
    \begin{align*}
    E_{[TS_{i_1}]}(-a_1^{(0)})\cdots E_{[TS_{i_m}]}(-a_m^{(0)})\in \mathbf{U}^-_{>0} \subseteq \mathbf{U}^+ n_0 \mathbf{B}^+.
    \end{align*} 
Recall that there exists a unique bijection $\phi:\mathbf{U}^+_{>0}\rightarrow \mathbf{U}^-_{>0}$ such that $\phi(u)\cdot \mathbf{B}^+ = u\cdot \mathbf{B}^-$ for any $u\in \mathbf{U}^+_{>0}$.
Then we have 
\begin{align*}
    &E_{[TS_{i_1}]}(-a_1^{(0)})\cdots E_{[TS_{i_m}]}(-a_m^{(0)}) \\
    =& E_{[\tau^{j_1}I_{i_1}]}(\varepsilon_1b_1) \cdots E_{[\tau^{j_m}I_{i_m}]}(\varepsilon_mb_m) n_0 b \\
    =& \phi^{-1}(E_{[TS_{i_1}]}(-a_1^{(0)})\cdots E_{[TS_{i_m}]}(-a_m^{(0)})) n_0 b',
\end{align*}
for some $b,b'\in \mathbf{B}^+$.
By Bruhat decomposition, the expression of the form $un_0b$ with $u\in \mathbf{U}^+,b\in \mathbf{B}^+$, should be unique.
This means 
\begin{align*}
    E_{[\tau^{j_1}I_{i_1}]}(\varepsilon_1b_1) \cdots E_{[\tau^{j_m}I_{i_m}]}(\varepsilon_mb_m) =\phi^{-1}(E_{[TS_{i_1}]}(-a_1^{(0)})\cdots E_{[TS_{i_m}]}(-a_m^{(0)})) \in \mathbf{U}^+_{>0},
\end{align*}
and thus the subset $\mathfrak{U}^+_{>0}$ is contained in $\mathbf{U}^+_{>0}$.

On the other hand, for any $u\in \mathbf{U}^+_{>0}$, there exists a unique $(a_1,\cdots,a_m)\in \mathbb{R}^m_{>0}$, such that $u=\phi^{-1}(E_{[TS_{i_1}]}(-a_1)\cdots E_{[TS_{i_m}]}(-a_m))$.
Let 
\begin{align*}
    (c_1,\cdots,c_m)=\varphi_{m-1}\cdots\varphi_0(a_1,\cdots,a_m)\in (\varepsilon_1\mathbb{R}_{>0})\times \cdots \times (\varepsilon_m \mathbb{R}_{>0}).
\end{align*}
Then by definition we can check that $(\varepsilon_1c_1,\cdots,\varepsilon_m c_m)\in \Omega_Q$.
By the same argument as above, we have 
\begin{align*}
    E_{[\tau^{j_1}I_{i_1}]}(c_1) \cdots E_{[\tau^{j_m}I_{i_m}]}(c_m)\cdot \mathbf{B}^- &= f_m(c_1,\cdots,c_m)\\
    &=f_m\varphi_{m-1}\cdots\varphi_0(a_1,\cdots,a_m)\\
    &=f_0(a_1,\cdots,a_m)\\
    &=E_{[TS_{i_1}]}(-a_1)\cdots E_{[TS_{i_m}]}(-a_m) \cdot \mathbf{B}^+,
\end{align*}
and thus 
\begin{align*}
    u=\phi^{-1}(E_{[TS_{i_1}]}(-a_1)\cdots E_{[TS_{i_m}]}(-a_m))=E_{[\tau^{j_1}I_{i_1}]}(c_1) \cdots E_{[\tau^{j_m}I_{i_m}]}(c_m) \in \mathfrak{U}^+_{>0}.
\end{align*}
The subset $\mathbf{U}^+_{>0}$ is contained in $\mathfrak{U}^+_{>0}$.

Hence $\mathfrak{U}^+_{>0}=\mathbf{U}^+_{>0}$, and the theorem is proved.
\end{proof}

This Theorem characterizes the size of the monoid $\mathbf{U}^+_{>0}$ of totally positive elements, with respect to the root subgroups corresponding to the indecomposable objects of the chosen complete hereditary subcategory $\mathcal{B}$ of $\mathcal{R}$.

Now we consider $\mathbf{U}^-_{>0}$.

Let $\tilde{\varepsilon}_1=1$, and for $2\leq k\leq m$, define $\tilde{\varepsilon}_k\in \{\pm 1\}$ such that 
\begin{align*}
    n_{i_1}\cdots n_{i_{k-1}} u_{S_{i_k}}=\tilde{\varepsilon}_k u_{\tau^{j_k}I_{i_k}}.
\end{align*}
We have 
\begin{align*}
    \tilde{\varepsilon}_k=\varepsilon_k (-1)^{A_{S_{i_1},\omega_{S_{i_2}}\cdots\omega_{S_{i_{k-1}}}(S_{i_k})}+A_{S_{i_2},\omega_{S_{i_3}}\cdots\omega_{S_{i_{k-1}}}(S_{i_k})}+\cdots+A_{S_{i_{k-1}},S_{i_k}}}.
\end{align*}

\begin{definition}
    Let $\mathfrak{U}^-_{>0}$ be the subset of $\mathbf{U}^-$ consisting of the following elements:
    \begin{align*}
        \mathfrak{U}^-_{>0}=\{ E_{[T\tau^{j_1}I_{i_1}]}(-\tilde{\varepsilon}_1b_1)\cdots  E_{[T\tau^{j_m}I_{i_m}]}(-\tilde{\varepsilon}_mb_m) \mid (b_1,\cdots,b_m)\in \Omega_Q \}.
    \end{align*}
\end{definition}

\begin{theorem}
    The subsets $\mathfrak{U}^-_{>0}=\mathbf{U}^-_{>0}$.
\end{theorem}

\begin{proof}
    Recall the following facts: $\eta_{X,TY}=\eta_{XY}, \omega_X(TY)=T\omega_X(Y)$.
    Thus we have $n_{[X]}E_{[TY]}(s)n_{[X]}^{-1}=E_{[T\omega_X(Y)]}(\eta_{XY}s)$.    
    Similar to the previous process, for $0\leq k\leq m$, we define a map $\tilde{f}_k:(\mathbb{R}^{\times})^m\rightarrow \mathfrak{B}$ by 
    \begin{align*}
        &\tilde{f}_k(a_1,\cdots,a_m)\\
        =& E_{[T\tau^{j_1}I_{i_1}]}(-a_1)\cdots E_{[T\tau^{j_k}I_{i_k}]}(-a_k)n_{i_1}\cdots n_{i_k} E_{[S_{i_{k+1}}]}(a_{k+1})\cdots E_{[S_{i_m}]}(a_m)\cdot \mathbf{B}^-.
    \end{align*}
    Then 
    \begin{align*}
         \tilde{f}_k&(a_1,\cdots,a_m)= E_{[T\tau^{j_1}I_{i_1}]}(-a_1)\cdots E_{[T\tau^{j_k}I_{i_k}]}(-a_k) E_{[T\tau^{j_{k+1}}I_{i_{k+1}}]}(-\tilde{\varepsilon}_{k+1}a_{k+1}^{-1})\times \\
        & \times  n_{i_1}\cdots n_{i_{k+1}} h_{i_{k+1}}(a_{k+1}^{-1}) E_{[TS_{i_{k+1}}]}(-a_{k+1}^{-1}) E_{[S_{i_{k+2}}]}(a_{k+2}) \cdots E_{[S_{i_m}]}(a_m)\cdot \mathbf{B}^-.
    \end{align*}
    
    For $k+1\leq s\leq m$, define $\tilde{f}_s':\mathbb{R}_{>0}^{m-k}\rightarrow\mathfrak{B}$ by 
    \begin{align*}
        &\tilde{f}_s'(a_{k+1},\cdots,a_m)\\
        =&E_{[S_{i_{k+2}}]}(a_{k+2})\cdots E_{[S_{i_s}]}(a_s)E_{[TS_{i_{k+1}}]}(-a_{k+1})E_{[S_{i_{s+1}}]}(a_{s+1})\cdots E_{[S_{i_m}]}(a_m) \cdot \mathbf{B}^-.
    \end{align*}
    By similar calculations to that in the proof of Proposition \ref{fk}, we obtain a map $\tilde{g}_k:\mathbb{R}_{>0}^{m-k}\rightarrow \mathbb{R}_{>0}^{m-k}$ such that $\tilde{f}_{k+1}'=\tilde{f}_m'\tilde{g}_k$.
    Moreover, it can be seen from the calculations that, for any $(a_{k+1},\cdots,a_m)\in\mathbb{R}_{>0}^{m-k}$, let $(d_{k+1},\cdots,d_m)=g_k(a_{k+1}^{-1},a_{k+2},\cdots,a_m)$, $(d_{k+1}',\cdots,d_m')=\tilde{g}_k(a_{k+1}^{-1},a_{k+2},\cdots,a_m)$, then we have $d_l=d_l'$ for $k+2 \leq l\leq m$. 
    Then 
    \begin{align*}
         \tilde{f}_k(a_1,\cdots,a_m)&= E_{[T\tau^{j_1}I_{i_1}]}(-a_1)\cdots E_{[T\tau^{j_k}I_{i_k}]}(-a_k) E_{[T\tau^{j_{k+1}}I_{i_{k+1}}]}(-\tilde{\varepsilon}_{k+1}a_{k+1}^{-1})\times \\
        & \times  n_{i_1}\cdots n_{i_{k+1}}   E_{[S_{i_{k+2}}]}(d_{k+2}a_{k+1}^{-A_{S_{i_{k+1}},S_{i_{k+2}}}}) \cdots E_{[S_{i_m}]}(d_ma_{k+1}^{-A_{S_{i_{k+1}},S_{i_m}}})\cdot \mathbf{B}^-.
    \end{align*}
    Thus 
    \begin{align*}
        \tilde{f}_k(a_1,\cdots,a_m)=\tilde{f}_{k+1}(a_1,\cdots,a_k,\tilde{\varepsilon}_{k+1}a_{k+1}^{-1},d_{k+2}a_{k+1}^{-A_{S_{i_{k+1}},S_{i_{k+2}}}},\cdots,d_ma_{k+1}^{-A_{S_{i_{k+1}},S_{i_m}}}).
    \end{align*}
    Define $\tilde{\varphi}_k:(\tilde{\varepsilon}_1\mathbb{R}_{>0})\times \cdots \times (\tilde{\varepsilon}_k\mathbb{R}_{>0})\times \mathbb{R}_{>0}^{m-k}\rightarrow (\tilde{\varepsilon}_1\mathbb{R}_{>0})\times \cdots \times (\tilde{\varepsilon}_{k+1}\mathbb{R}_{>0})\times \mathbb{R}_{>0}^{m-k-1}$ by 
    \begin{align*}
        \tilde{\varphi}_k(a_1,\cdots,a_m)=(a_1,\cdots,a_k,\tilde{\varepsilon}_{k+1}a_{k+1}^{-1},d_{k+2}a_{k+1}^{-A_{S_{i_{k+1}},S_{i_{k+2}}}},\cdots,d_ma_{k+1}^{-A_{S_{i_{k+1}},S_{i_m}}}).
    \end{align*}
    Then $\tilde{f}_k=\tilde{f}_{k+1}\tilde{\varphi}_k$.
    Recall that for $\varphi_k:(\varepsilon_1\mathbb{R}_{>0})\times \cdots \times (\varepsilon_k\mathbb{R}_{>0})\times \mathbb{R}_{>0}^{m-k}\rightarrow (\varepsilon_1\mathbb{R}_{>0})\times \cdots \times (\varepsilon_{k+1}\mathbb{R}_{>0})\times \mathbb{R}_{>0}^{m-k-1}$,
    we have 
    \begin{align*}
        &\varphi_k(a_1',\cdots,a_k',a_{k+1},\cdots,a_m)\\
        =&(a_1',\cdots,a_k',\varepsilon_{k+1}a_{k+1}^{-1},d_{k+2}a_{k+1}^{-A_{S_{i_{k+1}},S_{i_{k+2}}}},\cdots,d_ma_{k+1}^{-A_{S_{i_{k+1}},S_{i_m}}}).
    \end{align*}

    For any $(b_1,\cdots,b_m)\in \Omega_Q$, let $a_l^{(k)}$ be defined as in the proof of Theorem \ref{u=u}, and 
    \begin{align*}
        (\tilde{a}_1^{(m)},\cdots,\tilde{a}_m^{(m)})=(\tilde{\varepsilon}_1b_1,\cdots,\tilde{\varepsilon}_mb_m).
    \end{align*}
    Assume $\tilde{a}_l^{(k)}$ is defined for $1\leq l\leq m$, then clearly we have $\tilde{\varepsilon}_l a_l^{(k)}=\varepsilon_l \tilde{a}_l^{(k)}$ for $1\leq l\leq k$, and $\tilde{a}_l^{(k)}=a_l^{(k)}$ for $k+1\leq l \leq m$.
    Then 
    \begin{align*}
        0<\beta_k=b_k-\sum_{\substack{k+1<s\leq m\\i_s=i_k}}a_s^{(k)}=b_k-\sum_{\substack{k+1<s\leq m\\i_s=i_k}}\tilde{a}_s^{(k)}.
    \end{align*}
    By similar reasoning, we can define 
    \begin{align*}
        (\tilde{a}_1^{(k-1)},\cdots,\tilde{a}_m^{(k-1)})=\tilde{\varphi}_{k-1}^{-1}(\tilde{a}_1^{(k)},\cdots,\tilde{a}_m^{(k)})\in (\tilde{\varepsilon}_1\mathbb{R}_{>0})\times \cdots \times (\tilde{\varepsilon}_{k-1}\mathbb{R}_{>0})\times \mathbb{R}_{>0}^{m-k+1}
    \end{align*}
    and then proceed by induction.
    Finally, we have 
    \begin{align*}
        E_{[T\tau^{j_1}I_{i_1}]}(-\tilde{\varepsilon}_1b_1)\cdots  E_{[T\tau^{j_m}I_{i_m}]}(-\tilde{\varepsilon}_mb_m) \cdot \mathbf{B}^+ = E_{[S_{i_1}]}(a_1^{(0)})\cdots E_{[S_{i_m}]}(a_m^{(0)}) \cdot \mathbf{B}^-.
    \end{align*}
    By a proof completely analogous to that of Theorem \ref{u=u}, we have $\mathfrak{U}^-_{>0}=\mathbf{U}^-_{>0}$.
\end{proof}

Note that the characterization of $\mathbf{U}^+_{>0}$ and $\mathbf{U}^-_{>0}$ use the same region $\Omega_Q$, and the ordering of the indecomposables in $\mathfrak{U}^-_{>0}$ is exactly the $T$-shift of that in $\mathfrak{U}^+_{>0}$.

It's clear that
\begin{align*}
   \mathbf{H}_{>0}=\{ \prod_{i=1}^n h_i(t_i) \mid t_1,\cdots,t_n \in \mathbb{R}_{>0}  \}, 
\end{align*}
 and each element in $ \mathbf{G}_{>0}=\mathfrak{U}^+_{>0} \mathbf{H}_{>0} \mathfrak{U}^-_{>0}=\mathfrak{U}^-_{>0} \mathbf{H}_{>0} \mathfrak{U}^+_{>0}$ can be uniquely written of the form 
\begin{align*}
     E_{[\tau^{j_1}I_{i_1}]}(\varepsilon_1a_1)\cdots  E_{[\tau^{j_m}I_{i_m}]}(\varepsilon_ma_m) h_{[S_1]}(t_1)\cdots h_{[S_n]}(t_n)  E_{[T\tau^{j_1}I_{i_1}]}(-\tilde{\varepsilon}_1b_1)\cdots  E_{[T\tau^{j_m}I_{i_m}]}(-\tilde{\varepsilon}_mb_m)  
\end{align*}
or 
\begin{align*}
     E_{[T\tau^{j_1}I_{i_1}]}(-\tilde{\varepsilon}_1b_1')\cdots  E_{[T\tau^{j_m}I_{i_m}]}(-\tilde{\varepsilon}_mb_m')  h_{[S_1]}(t_1')\cdots h_{[S_n]}(t_n')   E_{[\tau^{j_1}I_{i_1}]}(\varepsilon_1a_1')\cdots  E_{[\tau^{j_m}I_{i_m}]}(\varepsilon_ma_m'),
\end{align*}
with $(a_1,\cdots,a_m),(a_1',\cdots,a_m'),(b_1,\cdots,b_m),(b_1',\cdots,b_m') \in \Omega_Q$, $t_1,\cdots,t_n,t_1',\cdots,t_n'\in\mathbb{R}_{>0}$.

\subsection{Description of $\mathbf{U}^+_{\geq 0}$, $\mathbf{U}^-_{\geq 0}$ and $\mathbf{G}_{\geq 0}$}
In this subsection, we will deal with the closures $\mathbf{U}^+_{\geq 0}$, $\mathbf{U}^-_{\geq 0}$ and $\mathbf{G}_{\geq 0}$ of $\mathbf{U}^+_{>0}$, $\mathbf{U}^-_{> 0}$ and $\mathbf{G}_{> 0}$, respectively.

The results in the previous subsection actually hold for any reduced expression for $\overline{n}_0$.
Precisely, let $(i_1,\cdots,i_m)$ be any reduced expression for $\overline{n}_0$.
Then the subset 
\begin{align*}
    \{ E_{[S_{i_1}]}(\varepsilon_1b_1) E_{[\omega_{S_{i_1}}(S_{i_2})]}(\varepsilon_2b_2)\cdots E_{[\omega_{S_{i_1}}\cdots\omega_{S_{i_{m-1}}}(S_{i_m})]}(\varepsilon_mb_m) \mid (b_1,\cdots,b_m)\in \Omega_{i_1,\cdots,i_m} \}
\end{align*}
coincides with $\mathbf{U}^+_{>0}$, and the subset 
\begin{align*}
    \{ E_{[TS_{i_1}]}(-\tilde{\varepsilon}_1b_1) E_{[T\omega_{S_{i_1}}(S_{i_2})]}(-\tilde{\varepsilon}_2b_2)\cdots E_{[T\omega_{S_{i_1}}\cdots\omega_{S_{i_{m-1}}}(S_{i_m})]}(-\tilde{\varepsilon}_mb_m) \mid (b_1,\cdots,b_m)\in \Omega_{i_1,\cdots,i_m} \}
\end{align*}
coincides with $\mathbf{U}^-_{>0}$, where 
\begin{align*}
    \tilde{\varepsilon}_k=\eta_{S_{i_{k-1}},S_{i_k}}\eta_{S_{i_{k-2}},\omega_{S_{i_{k-1}}}(S_{i_k})}\cdots \eta_{S_{i_1},\omega_{S_{i_2}}\cdots\omega_{S_{i_{k-1}}}(S_{i_k})}.
\end{align*}
The set $\mathfrak{U}^+_{>0}$ and $\mathfrak{U}^-_{>0}$ are special cases of these.

For any $n\in \mathbf{N}$, let $(i_{t+1},\cdots,i_m)$ be a reduced expression for $\overline{n}$.
There exists some $v\in \mathbf{N}$ such that $vn=n_0$.
Let $(i_1,\cdots,i_t)$ be a reduced expression for $\overline{v}$, then $(i_1,\cdots,i_m)$ is a reduced expression for $\overline{n}_0$.
For $t+1\leq k\leq m$, define  
\begin{align*}
    &\varepsilon_k^{\overline{n}}=\eta_{S_{i_{t+1}},\omega_{S_{i_{t+1}}}\cdots\omega_{S_{i_{k-1}}}(S_{i_k})} \cdots \eta_{S_{i_{k-1}},\omega_{S_{i_{k-1}}}(S_{i_k})},\\
    &\delta_k^{\overline{n}}=\eta_{S_{i_{1}},\omega_{S_{i_{1}}}\cdots\omega_{S_{i_{k-1}}}(S_{i_k})} \cdots \eta_{S_{i_{t}},\omega_{S_{i_t}}\cdots \omega_{S_{i_{k-1}}}(S_{i_k})},
\end{align*}
and we have $\varepsilon_k=\delta_k^{\overline{n}}\varepsilon_k^{\overline{n}}$.

Recall that for each $k$, $\beta_k(b_1,\cdots,b_m)$ can actually be determined by $b_k,\cdots,b_m$, so we may remove the irrelevant $b_i$'s in the expression and denote it, for example, by $\beta_k(b_k,\cdots,b_m)$.

\begin{definition}
    For a reduced expression $(i_{t+1},\cdots,i_m)$ of $\overline{n}$, define a region 
    \begin{align*}
        \Omega_{i_{t+1},\cdots,i_m}=\{ (b_{t+1},\cdots,b_m)\in \mathbb{R}_{>0}^{m-t} \mid \beta_k(b_{t+1},\cdots,b_m)>0, \forall k=t+1,\cdots,m \},
    \end{align*}
    which is the projection of $\Omega_{i_1,\cdots,i_m}$ onto the last $m-t$ components.
\end{definition}

Now we use this region to deal with $\mathbf{U}^-(\overline{n})$.

\begin{definition}
    Define a subset $\mathfrak{U}^-_{\mathfrak{B}}(\overline{n})$ of $\mathfrak{B}$ consisting of elements of the form 
\begin{align*}
  E_{[S_{i_{t+1}}]}(\varepsilon_{t+1}^{\overline{n}}b_{t+1}) E_{[\omega_{S_{i_{t+1}}}(S_{i_{t+2}})]}(\varepsilon_{t+2}^{\overline{n}}b_{t+2})  \cdots E_{[\omega_{S_{i_{t+1}}}\cdots\omega_{S_{i_{m-1}}}(S_{i_m})]}(\varepsilon_m^{\overline{n}}b_m)n_{i_{t+1}}^{-1}\cdots n_{i_m}^{-1}\cdot \mathbf{B}^+ 
\end{align*}
with $ (b_{t+1},\cdots,b_m)\in \Omega_{i_{t+1},\cdots,i_m} $.
\end{definition}

\begin{proposition}
    The two subsets $\mathfrak{U}^-_{\mathfrak{B}}(\overline{n})$ and $\mathbf{U}^-(\overline{n})\cdot \mathbf{B}^+$ of $\mathfrak{B}$ are equal.

    Moreover, for each $(b_{t+1},\cdots,b_m)\in \Omega_{i_{t+1},\cdots,i_m}$, there exists a unique element $\mathbf{b}(b_{t+1},\cdots,b_m)$ in $\mathbf{B}^+$, such that the set of elements
    \begin{align*}
        E_{[S_{i_{t+1}}]}(\varepsilon_{t+1}^{\overline{n}}b_{t+1})  \cdots E_{[\omega_{S_{i_{t+1}}}\cdots\omega_{S_{i_{m-1}}}(S_{i_m})]}(\varepsilon_m^{\overline{n}}b_m)n_{i_{t+1}}^{-1}\cdots n_{i_m}^{-1} \mathbf{b}(b_{t+1},\cdots,b_m) 
    \end{align*}
    with $ (b_{t+1},\cdots,b_m)\in \Omega_{i_{t+1},\cdots,i_m} $ coincides with $\mathbf{U}^-(\overline{n})$.
\end{proposition}

\begin{proof}
    For $(b_1,\cdots,b_m)\in \mathbb{R}_{>0}^m$, let $(a_1^{(m)},\cdots,a_m^{(m)})=(\varepsilon_1b_1,\cdots,\varepsilon_mb_m)$.
We have 
\begin{align*}
    & E_{[S_{i_1}]}(a_1^{(m)})\cdots E_{[\omega_{S_{i_1}}\cdots\omega_{S_{i_{m-1}}}(S_{i_m})]}(a_m^{(m)})n_{i_1}^{-1}\cdots n_{i_m}^{-1}\cdot \mathbf{B}^+\\
   =& E_{[S_{i_1}]}(a_1^{(m)})\cdots E_{[\omega_{S_{i_1}}\cdots\omega_{S_{i_{t-1}}}(S_{i_t})]}(a_t^{(m)})n_{i_1}^{-1}\cdots n_{i_t}^{-1} \times \\
   & \times E_{[S_{i_{t+1}}]}(\delta_{t+1}^{\overline{n}}a_{t+1}^{(m)})\cdots E_{[\omega_{S_{i_{t+1}}}\cdots\omega_{S_{i_{m-1}}}(S_{i_m})]}(\delta_m^{\overline{n}}a_m^{(m)}) n_{i_{t+1}}^{-1}\cdots n_{i_m}^{-1}   \cdot \mathbf{B}^+
\end{align*}
Assume that $(b_{t+1},\cdots,b_m)\in \Omega_{i_{t+1},\cdots,i_m}$. 
Using the notations as in the previous subsection, we have $a_{l}^{(k)}$ well-defined, for $1\leq l\leq m$ and $t\leq k\leq m$.
Since $\delta_k^{\overline{n}}a_k^{(m)}=\delta_k^{\overline{n}}\varepsilon_kb_k=\varepsilon_k^{\overline{n}}b_k$ for $t+1\leq k \leq m$, we have 
\begin{align*}
          &E_{[S_{i_{t+1}}]}(\varepsilon_{t+1}^{\overline{n}}b_{t+1})  \cdots E_{[\omega_{S_{i_{t+1}}}\cdots\omega_{S_{i_{m-1}}}(S_{i_m})]}(\varepsilon_m^{\overline{n}}b_m)n_{i_{t+1}}^{-1}\cdots n_{i_m}^{-1} \cdot \mathbf{B}^+ \\
    =&E_{[TS_{i_{t+1}}]}(-a_{t+1}^{(t)})\cdots E_{[TS_{i_m}]}(-a_m^{(t)}) \cdot \mathbf{B}^+.
\end{align*}
Note that $(a_{t+1}^{(t)},\cdots,a_m^{(t)})\in \mathbb{R}_{>0}^{m-t}$, so $E_{[TS_{i_{t+1}}]}(-a_{t+1}^{(t)})\cdots E_{[TS_{i_m}]}(-a_m^{(t)})\in \mathbf{U}^-(\overline{n})$.
We have  $\mathfrak{U}^-_{\mathfrak{B}}(\overline{n}) \subseteq \mathbf{U}^-(\overline{n})\cdot \mathbf{B}^+$.
On the other hand, by Lemma \ref{varphikbij}, we have $\mathbf{U}^-(\overline{n})\cdot \mathbf{B}^+  \subseteq \mathfrak{U}^-_{\mathfrak{B}}(\overline{n})$, and thus these two sets are equal.

Since the correspondence between
\begin{align*}
    (b_{t+1},\cdots,b_m)\in \Omega_{i_{t+1},\cdots,i_m}\quad \text{and} \quad (a_{t+1}^{(t)},\cdots,a_m^{(t)})\in \mathbb{R}_{>0}^{m-t}
\end{align*}
 is bijective, the second assertion follows.
\end{proof}

Denote the set of elements of the form
\begin{align*}
        E_{[S_{i_{t+1}}]}(\varepsilon_{t+1}^{\overline{n}}b_{t+1})  \cdots E_{[\omega_{S_{i_{t+1}}}\cdots\omega_{S_{i_{m-1}}}(S_{i_m})]}(\varepsilon_m^{\overline{n}}b_m)n_{i_{t+1}}^{-1}\cdots n_{i_m}^{-1} \mathbf{b}(b_{t+1},\cdots,b_m), 
\end{align*}
with $ (b_{t+1},\cdots,b_m)\in \Omega_{i_{t+1},\cdots,i_m} $ by $\mathfrak{U}^-(\overline{n})$.
The previous Proposition tells that $\mathfrak{U}^-(\overline{n})=\mathbf{U}^-(\overline{n})$ and $\mathfrak{U}^-(\overline{n})\cdot \mathbf{B}^+ = \mathfrak{U}^-_{\mathfrak{B}}(\overline{n})$.
We can similarly define $\mathfrak{U}^+(\overline{n})$ with analogous properties.

Finally, we obtain the following theorem.

\begin{theorem}
    We have 
    \begin{align*}
   &\mathbf{U}^+_{\geq 0}=\bigsqcup_{\overline{n}\in \mathbf{N}/\mathbf{H}} \mathfrak{U}^+(\overline{n}), \qquad 
    \mathbf{U}^-_{\geq 0}=\bigsqcup_{\overline{n}\in \mathbf{N}/\mathbf{H}} \mathfrak{U}^-(\overline{n}),\\
    &\mathbf{G}_{\geq 0}=\bigsqcup_{\overline{n},\overline{n}'\in \mathbf{N}/\mathbf{H}} \mathfrak{U}^+(\overline{n})\mathbf{H}_{>0}\mathfrak{U}^-(\overline{n}')=\bigsqcup_{\overline{n},\overline{n}'\in \mathbf{N}/\mathbf{H}}\mathfrak{U}^-(\overline{n}')\mathbf{H}_{>0}\mathfrak{U}^+(\overline{n}).
\end{align*}
\end{theorem}

\subsection*{Acknowledgments}
We would like to thank Jiepeng Fang and Yixin Lan for their discussion.
The authors were partially supported by National Natural Science Foundation of China [Grant No. 12471030].

\end{document}